\documentclass[11pt, a4paper, reqno, english]{smfart}

\usepackage[margin=3.5cm]{geometry} 
\usepackage{amsfonts, amsthm, amsmath, amscd, amssymb}

\renewcommand\AA{\mathbb{A}}
\newcommand\A{\mathcal{A}}
\newcommand\QQ{\mathbb{Q}}
\newcommand\cQ{\overline{\mathbb{Q}}}

\newcommand\NN{\mathbb{N}}
\newcommand\RR{\mathbb{R}}

\newcommand\WW{\mathcal{W}}

\newcommand\ZZ{\mathbb{Z}}

\newcommand\PP{\mathbb{P}}
\newcommand\x{\mathbf{x}}
\newcommand\y{\mathbf{y}}
\renewcommand\v{\mathbf{v}}
\newcommand\z{\mathbf{z}}

\renewcommand\d{\,\mathrm{d}}

\DeclareMathOperator{\rank}{rank}

\DeclareMathOperator{\sing}{sing}

\DeclareMathOperator{\Gal}{Gal}
\DeclareMathOperator{\Tr}{Tr}
\DeclareMathOperator{\Norm}{N}

\renewcommand{\leq}{\leqslant}
\renewcommand{\geq}{\geqslant}

\newcommand\ve{\varepsilon}

\newcommand\al{\alpha}

\newcommand\D{\Delta}

\newtheorem{theorem}{Theorem}
\newtheorem*{theorem-f}{Th\'eor\`eme}
\newtheorem*{cor}{Corollary}
\newtheorem*{cor-f}{Corollaire}
\newtheorem*{con}{Conjecture}
\newtheorem{lemma}{Lemma}

\theoremstyle{definition}
\newtheorem{remark}{Remarque}

\newtheorem*{ack}{Acknowledgements}
\newtheorem*{notat}{Notation}

\newcommand{\M}{\mathfrak{M}}
\renewcommand{\ss}{\mathfrak{S}}
\newcommand{\ii}{\mathfrak{I}}
\newcommand{\n}{\mathfrak{n}}
\newcommand{\m}{\mathfrak{m}}

\numberwithin{equation}{section}

\newcommand\kbar{{\overline k}}

\newcommand\pic{{\rm Pic} \hskip1mm}
\newcommand\br{{\rm Br} \hskip1mm}
\newcommand \Z{{\mathbf Z}}
\newcommand\g{{\mathcal G}}
\newcommand\Div{{\rm Div}  }
\newcommand\X{{  \overline{X} }  }
\newcommand\Y{{  \overline{Y} }  }
\newcommand\U{{  \overline{U} }  }
\newcommand\et{\mathrm{\mathaccent 19 et}}
\newcommand\G{{\mathbf G}}

\begin{document}

\title[Rational points on cubic hypersurfaces]{Rational points on cubic hypersurfaces \\that
  split off a form}

\author{T.D. Browning}
\address{School of Mathematics\\
University of Bristol\\ Bristol BS8 1TW}
\email{t.d.browning@bristol.ac.uk}

\date{\today}

\begin{abstract}
Let $X$ be a projective  
cubic hypersurface of dimension $11$ or more,
which is defined over $\QQ$. We show that $X(\QQ)$ is non-empty provided
that the cubic form defining $X$ can be written as the sum of two 
forms that share no common variables.
\end{abstract}


\maketitle
\tableofcontents

\section{Introduction}

Let $X\subset\PP^{n-1}$ be a cubic
hypersurface, given as the zero locus of a cubic
form $C\in\ZZ[x_1,\ldots,x_n]$. This paper is concerned with the
problem of determining when the set of rational points $X(\QQ)$ on $X$ is
non-empty.  There is a well-known conjecture that $X(\QQ)\neq \emptyset$ as soon as
$n \geq 10$.  
In fact, for non-singular cubic hypersurfaces, it is expected that the
Hasse principle holds as soon as $n\geq 5$. This 
states that in order for $X(\QQ)$ to be non-empty it is necessary and
sufficient that $X(\QQ_p)$ is non-empty for every prime $p$. 
For a large class of possibly singular cubic hypersurfaces $X\subset \PP^{n-1}$, 
Salberger has calculated the Brauer group
$\mathrm{Br}(Y)$ associated to a projective non-singular model $Y$ of
$X$.  A detailed proof of this calculation is provided by 
Colliot-Th\'el\`ene in the appendix to this
paper. As a consequence of this investigation
one has the following prediction.

\begin{con}
Let $X\subset\PP^{n-1}$ be a cubic
hypersurface defined over $\QQ$ which is not a cone, with $n\geq 5$ and 
singular locus which is empty or of codimension at least $4$ in $X$.
Then the Hasse principle holds for the locus of non-singular points on $X$.
\end{con}

Let us now consider some of the progress that has been made towards
this conjecture. 
When $C$ is diagonal it follows from the work
of Baker \cite{baker} that $X$ has $\QQ$-rational points as soon as
$n\geq 7$. At the opposite end of the spectrum, when absolutely no assumptions are
made about the shape of $C$, a lot of work has been invested in producing a
reasonable lower bound for the number of variables needed to ensure that
$X(\QQ)\neq \emptyset$. Building on work of Davenport \cite{dav-32, dav-16},  Heath-Brown \cite{14} has 
recently shown that $n\geq 14$ variables are enough to
secure this property for an arbitrary cubic
hypersurface defined over $\QQ$.
In the light of this body of work it is very natural to try and evince
intermediate 
results in which the existence of rational points is guaranteed
for cubic hypersurfaces in fewer than $14$
variables when mild assumptions are made about the structure of the
hypersurface. It is precisely this point of view that is the focus of the present investigation. 

Let $\sing(X)$ denote the singular locus of $X$,
as a projective subvariety of $X$. When $C\in \ZZ[x_1,\ldots,x_n]$ is
non-singular, so that  $\sing(X)$ is empty, it follows from work of
Hooley \cite{hooley1} that the Hasse principle holds for $X$
provided that $n\geq 9$.
As is well-known, the local conditions are automatic when $n\geq 10$,
and so $X(\QQ)$ is non-empty for non-singular $X$ provided that $n\geq
10$.  This fact was first proved by Heath-Brown \cite{hb-10}. 
When $\sing(X)$ has dimension
$\sigma\geq 0$, joint work of the author \cite{roth} with Heath-Brown 
shows that $X(\QQ)$ is non-empty provided that 
$$
n\geq 
\begin{cases}
11, & \mbox{if $\sigma=0$,}\\
12, & \mbox{if $\sigma=1$,}\\
13, & \mbox{if $\sigma=2$.}
\end{cases}
$$
We will make use of this result shortly.

Let $m<n$ be a positive integer. 
We will say that an integral cubic form $C$ in $n$ variables ``splits off an
$m$-form'' if there exist non-zero cubic forms $C_1,C_2$
with integer coefficients so that
$$
C(x_1,\ldots,x_n)=C_1(x_1,\ldots,x_m)+C_2(x_{m+1},\ldots,x_n),
$$
identically in $x_1,\ldots,x_n$. We will merely say that 
$C$ ``splits off a form'' if $C$ splits off an $m$-form for some
$1\leq m < n$.  The following is our 
main result.

\begin{theorem}\label{main}
Let $X\subset \PP^{n-1}$ be a hypersurface defined by a cubic
form that splits off a form, with $n\geq 13$. Then $X(\QQ)\neq \emptyset$.
\end{theorem}

The essential content of Theorem \ref{main} is that we can save $1$
variable in the result of Heath-Brown \cite{14} when the
underlying cubic form splits off a form.
It should be stressed that 
the existence of a single $\QQ$-rational point on $X$ is
enough to demonstrate the $\QQ$-unirationality of $X$, 
and so the Zariski density of $X(\QQ)$ in $X$, when $X$ is
geometrically integral and not a
cone.   Variants of this result have been known for a long time (cf
\cite{ct2,manin, segre-b}). 
In the generality with which we have stated the result, it appears in 
the work of Colliot-Th\'el\`ene and Salberger
\cite[Proposition 1.3]{c-t-s} and in that of Koll\'ar \cite{kollar}.

Our work has implications for the problem of
determining when an arbitrary cubic form $C\in \ZZ[x_1,\ldots,x_n]$ represents 
all non-zero $a\in \QQ$, using rational values for the variables. 
When this property holds we say that $C$
``captures $\QQ^*$''.  Recall that a cubic form is said to be degenerate if 
the corresponding cubic hypersurface is a cone.
Fowler \cite{fowler} has shown that any
non-degenerate cubic
form that  represents $0$ automatically captures $\QQ^*$ provided only
that $n\geq 3$.  Hence it suffices to fix attention on those forms
that do not represent zero non-trivially. On multiplying through by
denominators it will clearly suffice to establish 
that cubic forms of the shape 
\begin{equation}
  \label{eq:1-form}
C(x_1,\ldots,x_n)-ax_{n+1}^3 
\end{equation}
represent zero non-trivially, with $a$ an
arbitrary non-zero integer.  But this form splits off a
$1$-form, and so is handled by  Theorem \ref{main}. In
this way we deduce the following result.

\begin{cor}
Let $C\in \ZZ[x_1,\ldots,x_n]$ be a non-degenerate cubic form, with
$n\geq 12$. 
Then $C$ captures $\QQ^*$.
\end{cor}

This result should be compared with the work of Heath-Brown \cite{14}, which
implies that $n\geq 13$ variables
suffice. It follows from the work of 
Hooley \cite{hooley1} that this may be improved to $n\geq 8$ 
when $C$ is non-singular. As indicated in \cite[Appendix 1]{hb-10} the
latter lower bound is 
probably the correct one for arbitrary cubic forms, 
since \eqref{eq:1-form} always has
non-trivial $p$-adic zeros for $n$ in this range.

Let $n\geq 4$. 
When $X\subset \PP^{n-1}$ is a hypersurface defined by a cubic
form that splits off a form, we are able to handle fewer 
variables when appropriate assumptions are made about one of the forms. 
If $X$ is a cone then we will see in Lemma \ref{lem:p1.1} that
$X(\QQ)\neq \emptyset$. If, on the other hand, $X$ is not a cone let
us consider the effect of supposing that 
the underlying cubic form splits off a
non-singular $m$-form $C_1(x_1,\ldots,x_m)$.
If $m=n-1$ then $X$ is itself non-singular and we automatically have
$X(\QQ)\neq \emptyset$ when $n\geq 10$. If $m\leq n-2$ then 
the residual form $C_2$ defines a projective cubic hypersurface of dimension $n-m-2$, and as such
has singular locus of dimension at most $n-m-3$.
But then it follows that $X$ has singular locus
of dimension at most $n-m-3$. Thus 
we may deduce from \cite{roth} that
$X(\QQ)\neq \emptyset$ provided that $n\geq 8+n-m$. 
We record this observation in the following result.

\begin{theorem}\label{t:goat}
Let $X\subset \PP^{n-1}$ be a hypersurface defined by a cubic
form that splits off a non-singular $m$-form, with $m\geq 8$ and
$n\geq 10$.
Then $X(\QQ)\neq \emptyset$.
\end{theorem}

It would be interesting to reduce the range of $m$ needed to ensure
the validity of Theorem \ref{t:goat}. 
Ours is not the first attempt to better understand the arithmetic of
cubic hypersurfaces that split off a form. 
Indeed, in Davenport's  \cite{dav-16} treatment of cubic forms in $16$ variables,
a fundamental ingredient in the treatment of certain
bilinear equations is a separate analysis of those forms that split
into two.  In further work, 
Colliot-Th\'el\`ene and Salberger \cite{c-t-s} have shown that the
Hasse principle holds for any cubic hypersurface in $\PP^{n-1}$ that contains a set of 
three conjugate singular points, provided only that $n\geq 4$.
Given a cubic extension $K$ of $\QQ$, define the corresponding norm form 
\begin{equation}
  \label{eq:NORM}
N(x_1,x_2,x_3):=\Norm_{K/\QQ}(\omega_1 x_1+\omega_2 x_2+\omega_3 x_3),
\end{equation}
where $\{\omega_1, \omega_2, \omega_3\}$ is a basis of $K$ as a vector
space over $\QQ$. In view of the fact that the local conditions are
automatically satisfied for cubic forms in at least $10$ variables, we
observe the following easy consequence.

\begin{theorem}[Colliot-Th\'el\`ene and Salberger 
\cite{c-t-s}]
\label{main'''}
Let $X\subset \PP^{n-1}$ be a hypersurface defined by a cubic
form that splits off a norm form, with $n\geq 10$. Then $X(\QQ)\neq \emptyset$.
\end{theorem}

It turns out that Theorem \ref{main'''} will play a
useful r\^ole in dispatching some of the cases that arise in the proof
of Theorem \ref{main}.   Following the strategy of 
Birch, Davenport and Lewis \cite{BDL}, 
it would however be straightforward to adapt the proof of Theorem~\ref{main}
to retrieve Theorem \ref{main'''}.

An obvious further line of enquiry would be to investigate cubic hypersurfaces that split
off two forms, by which we mean that the corresponding cubic form can be written
as 
$$
C(x_1,\ldots,x_n)=C_1(x_1,\ldots,x_\ell)+C_2(x_{\ell+1},\ldots,x_m)+C_3(x_{m+1},\ldots,x_n),
$$
identically in $x_1,\ldots,x_n$, for appropriate $1\leq \ell< m <n$.  With the extra structure apparent in
such hypersurfaces one would like to determine the most
general conditions possible under which the conjectured value of $n\geq 10$ variables suffices to
ensure the existence of $\QQ$-rational points.

One of the remarkable features of our argument is
the breadth of tools that it draws upon. The underlying machinery is
the Hardy--Littlewood circle method,  and we certainly take advantage
of many of the contributions to the theory
of polynomial cubic exponential sums that have been made during the
last fifty years.  These are detailed in \S \ref{sec:exp_sums}.
A further component of our work 
involves a detailed analysis of the case in which one of the forms
that splits off in Theorem \ref{main} is singular and has a 
relatively small number of variables. To deal with this scenario 
it pays to reflect upon the classification of singular cubic hypersurfaces. This is a very old
topic in algebraic geometry, and can be traced back to the pioneering
work of Cayley \cite{cayley} and Schl\"afli \cite{cayley'}.
All of the necessary information will be collected together in \S \ref{s:geom}.
The final ingredient in  our work comprises good upper bounds for the
number of $\QQ$-rational points of bounded height on auxiliary cubic
hypersurfaces.  The estimates that we will take advantage of are
presented in \S \ref{sec:upper}.

When it is applicable, the Hardy--Littlewood circle method allows
us to show that $X(\QQ)\neq\emptyset$ for a given cubic hypersurface $X\subset
\PP^{n-1}$ by evaluating asymptotically the number of $\QQ$-rational points of bounded
height on $X$. It is a well-known but intriguing feature of the method that 
one can achieve such precise information by first establishing weaker
upper bounds for the growth rate of $\QQ$-rational points on
appropriate auxiliary varieties. 
In fact, we will show in Lemma \ref{lem:8} that the 
$\QQ$-rational points on a non-singular cubic
hypersurface $X\subset \PP^{n-1}$ satisfy the growth bound
$$
\#\{x\in X(\QQ): H(x)\leq P\} \ll_{\ve, X} P^{\dim X-\frac{1}{2}+\ve},
$$
for any $P\geq 1$,  provided that $\dim X\geq 6$.  
Here $H: \PP^{n-1}(\QQ)\rightarrow \RR_{\geq 0}$ is the
usual exponential height function.  This should be compared with the
Manin conjecture \cite{f-m-t} which predicts that the 
exponent of $P$ should be  $\dim X-1$ as soon as $\dim X\geq 3$.

\begin{notat}
Throughout our work $\NN$ will denote the set of positive
integers.  For any $\al\in \RR$, we will follow common convention and
write $e(\al):=e^{2\pi i\al}$ and $e_q(\al):=e^{2\pi i\al/q}$. The
parameter $\ve$ will always denote a very small positive real
number.  We will use $|\x|$ to denote the norm $\max |x_i|$ 
of a vector $\x=(x_1,\ldots,x_n)\in \RR^n$, whereas $\|\x\|$ will
be reserved for the usual Euclidean norm $\sqrt{x_1^2+\cdots+x_n^2}$. 
All of the implied constants that appear in this 
work will be allowed to
depend upon the coefficients of the cubic forms under
consideration and the
parameter $\ve>0$.  Any further dependence will be explicitly
indicated by appropriate subscripts.
\end{notat}

\begin{ack}
It is a pleasure to thank 
Professor Colliot-Th\'el\`ene, Professor Heath-Brown, Professor Salberger and 
Professor Wooley for several useful discussions.  
While working on this paper the author was supported by EPSRC
grant number \texttt{EP/E053262/1}.
\end{ack}

\section{Geometry of singular cubic hypersurfaces}\label{s:geom}

The proof of Theorem \ref{main} will depend intimately on the  
dimension of the hypersurfaces defined by the constituent cubic
forms, and the nature of their singularities. 
A key step will be to determine conditions on this
singular locus under which the hypersurface 
automatically has rational points.

Let $C\in \ZZ[x_1,\ldots,x_n]$ be an arbitrary cubic form, which
we assume takes the shape
\begin{equation}
  \label{eq:C}
C(\x):=\sum_{i,j,k}c_{ijk}x_ix_jx_k,
\end{equation}
in which the coefficients $c_{ijk}\in \ZZ$ are symmetric in the indices $i,j,k$.
Define the $n\times n$ matrix $M(\x)$ with $j,k$-entry
$\sum_i c_{ijk}x_i.$  We will say that the cubic form $C$ is ``good'' if 
for any $H\geq 1$ and any $\ve>0$ we have the upper bound
$$
\#\{\x\in\ZZ^{n}: |\x|\leq H,\,\rank M(\x)=r\}\ll H^{r+\ve},
$$
for each integer $0\leq r\leq n$.  A crucial step in Davenport's \cite{dav-16}
treatment of general cubic forms is a proof of the fact that forms
that fail to be good automatically possess non-trivial integer
solutions for ``geometric reasons''.  
Our approach has a similar flavour, although the underlying arguments will
be more obviously geometric.

Assume throughout this section that $n \geq 3$ and $X\subset
\PP^{n-1}$ is a hypersurface defined by a cubic form $C\in \ZZ[x_1,\ldots,x_n]$.
A lot of the facts
that we will record are classical. 
Suppose for the moment that $C$ is not absolutely irreducible. Then
either it has a linear factor $L$ defined over $\QQ$, or it is a product
$C=L_1L_2L_3$ of three linear factors that are conjugate over $\cQ$.
By considering the equation $L=0$ in the former case, 
we deduce that $X(\QQ)\neq\emptyset$. 
In the latter case, we arrive at the same conclusion when $n\geq 4$, by
considering the system of equations $L_1=L_2=L_3=0$. 
When $n=3$ and $C$ is a product of three conjugate factors we deduce
that $X$ has precisely three conjugate singular points.  
When $n\geq 3$ and $X$ is defined by an absolutely
irreducible cubic form, but is a cone, we note that the space of
vertices on $X$ must be a linear space globally defined over $\QQ$.
Thus $X(\QQ)\neq \emptyset$ in this case too. 
We have therefore established the following simple result.

\begin{lemma}\label{lem:p1.1}
Let $n \geq 4$. If $X$ is not geometrically integral, or if
$X$ is a cone, then $X(\QQ)\neq \emptyset$. When $n=3$ the same
conclusion holds unless $X$ contains precisely three conjugate singular points.
\end{lemma}

Recall that a cubic hypersurface $X$ is said to 
be non-singular if over $\cQ^n$ the only solution to the
system of equations $\nabla C(\x)=\mathbf{0}$ has $\x=\mathbf{0}$. 
Henceforth we will be predominantly interested in 
singular cubic forms, and then only in the cases $n=3,4$ and $5$.
Let $k$ be a field. It has been conjectured by 
Cassels and Swinnerton-Dyer that any cubic hypersurface $X \subset \PP^{n-1}$ defined over $k$ that contains a
$k$-rational $0$-cycle of degree coprime to $3$, automatically has
a $k$-rational point.  The case $n=3$ goes back to Poincar\'e. When
the singular locus is non-empty, the case $n=4$  can be 
deduced from the work of Skolem \cite{skolem}. A comprehensive
discussion of the arithmetic of singular cubic surfaces can be found
in the work of Coray and Tsfasman \cite{c-t}. 
Coray \cite{coray76b} has established the conjecture for all local
fields and,  in a subsequent investigation \cite[Proposition 3.6]{coray87}, has also dispatched
the case in which  $n=5$ and the $0$-cycle is made up of double points.

Suppose first that $n=3$, so that $X\subset \PP^2$ defines a cubic
curve, which we assume to be geometrically integral and not a cone. 
When $X$ is singular it contains exactly one 
singular point, which must therefore be defined over $\QQ$. Once
combined with Lemma \ref{lem:p1.1} we arrive at the following result.

\begin{lemma}\label{lem:curve}
Let $n=3$ and suppose that $X(\QQ)= \emptyset$. Then 
\begin{enumerate}
\item[(i)]
$X$ is non-singular; or
\item[(ii)]
$X$ contains precisely three conjugate singular points. 
\end{enumerate}
\end{lemma}

In case (ii) of Lemma \ref{lem:curve} one concludes that
the underlying cubic form can be written as a norm form \eqref{eq:NORM},
for appropriate $\omega_1,\omega_2,\omega_3 \in K$, 
where $K$ is the cubic number field obtained by adjoining one
of the singularities.

We turn now to the case $n=4$ of cubic surfaces 
$X\subset \PP^3$, which we suppose to be geometrically integral and
not equal to a cone. Suppose that $X$ is singular. The classification
of such cubic surfaces can be traced back to Cayley \cite{cayley} and Schl\"afli \cite{cayley'}, but we
will employ the modern treatment found in the work of Bruce and
Wall \cite{b-w}. In particular the singular locus of $X$ is either a
single line, in which case $X$ is ruled, or else it contains $\delta\leq 4$
isolated singularities and these are all rational double points. 
It follows that $X(\QQ)\neq \emptyset$ unless $\delta=3$ and the
three singular points are conjugate to each other over a cubic extension of $\QQ$.
In this final case, Skolem \cite{skolem} showed that
$C$ can be written as
\begin{equation}
  \label{eq:4_norm}
\Norm_{K/\QQ}(x_1\omega_1+x_2\omega_2+x_3\omega_3)+ax_4^2\Tr_{K/\QQ}(x_1\omega_1+x_2\omega_2+x_3\omega_3)+bx_4^3,
\end{equation}
for appropriate coefficients $\omega_1,\omega_2,\omega_3 \in K$ and
$a,b \in \ZZ$, where $K$ is the cubic number field obtained by adjoining one
of the singularities to $\QQ$. In terms of the classification over $\cQ$ 
according to singularity type, the 
only possibility here is that $X$ has
singularity type $3\mathbf{A}_i$ for $i=1$ or $2$, since the action of
$\Gal(\cQ/\QQ)$ preserves the singularity type.
Bringing this all together, we have therefore established the
following analogue of Lemma \ref{lem:curve}.

\begin{lemma}\label{lem:surface}
Let $n=4$ and suppose that $X(\QQ)= \emptyset$. Then 
\begin{enumerate}
\item[(i)] 
$X$ is non-singular; or
\item[(ii)]
$X$ contains precisely three conjugate double points. 
\end{enumerate}
\end{lemma}

We now try to construct a version of Lemmas \ref{lem:curve} and
\ref{lem:surface} for the case $n=5$. 
Let $Y\subset X$ denote the singular locus of $X\subset \PP^4$, a variety of
dimension at most $2$.  As usual we assume that $X$ is geometrically
integral and not a cone.  
We analyse $Y$ by considering the intersection of $X$ with a generic
hyperplane $H\in {\PP^4}^*$.  
In particular the hyperplane section 
$$
S_H=H \cap X
$$ 
is a geometrically integral cubic surface which is not a cone (see
\cite[Proposition 18.10]{harris}, for example).
In taking $H$ to be defined over $\QQ$, we may further assume that
$S_H$ is defined over $\QQ$. Any $\QQ$-rational point on $S_H$ visibly
produces a $\QQ$-rational point on $X$. 
Let $T_H  \subset S_H$ denote the singular locus of $S_H$. 
Then the classification of cubic surfaces 
implies that $T_H$ is either empty or it is a union of $\delta_H\leq
4$ points or it is a line. 
When $T_H$ is non-empty it follows from Lemma~\ref{lem:surface}
that either $S_H(\QQ)\neq \emptyset$ or else $T_H$ is finite, with 
$\delta_H=3$ and the three points being conjugate
to each other over $\cQ$.

Now an application of Bertini's theorem (in the form given by Harris 
\cite[Theorem~$17.16$]{harris}, for example) shows that
$$
H\cap Y =T_H.
$$
When $S_H$ is non-singular it therefore follows that $H\cap Y$ is
empty for generic   $H\in {\PP^4}^*$, whence $Y$ must be finite. In
the alternative case, when
$S_H$ is singular, we may conclude that $\#(H\cap Y)=3$ for 
generic   $H\in {\PP^4}^*$, whence the maximal component of $Y$
is a cubic curve.

Let us examine further the possibility that the singular locus $Y$ of
$X$ has dimension $1$, and that it contains a cubic curve $Y_0$ as its
component of maximal dimension. Clearly $Y_0$ is defined over $\QQ$.
Furthermore,  we may conclude from B\'ezout's theorem that the line connecting
any two points of $Y_0$ must be contained in $X$, since each such
point is a singularity of $X$.

If $Y_0$ is reducible over $\QQ$ then there are two
basic possibilities: either it is a union of lines or it is a union of
a conic and a line.  In the latter case $Y_0$ contains a line defined
over $\QQ$ and it trivially follows
that $Y_0(\QQ)\neq \emptyset$.  The former case fragments into a number
of subcases: either it is a union of $3$ concurrent lines,  or it contains
a pair of skew lines, or it is a union of $3$ coplanar lines, or it
contains a repeated line. The second case is 
impossible since then the join of the two skew lines defines a
$3$-plane  that would also be contained in $X$, contradicting the fact that $X$ is geometrically
irreducible.  It follows from consideration of the Galois
action on $Y_0$ that $Y_0(\QQ)\neq \emptyset$ in every case apart from
the one in which $Y_0$ is a union of $3$ coplanar lines.

If $Y_0$ is geometrically irreducible then it cannot be a twisted
cubic since then the secant variety $S(Y_0)\cong
\PP^3$ would be contained in $X$. 
Our argument so far has shown that either $Y_0(\QQ)\neq \emptyset$ for trivial
reasons, or else $Y_0$ is a cubic plane curve that is either geometrically
irreducible or a union of $3$ distinct lines.
The plane $P$ containing $Y_0$ is defined over $\QQ$ and, after carrying out a
linear change of variables, we may take
$x_1=x_2=0$ as its defining equations. But then it follows
that the cubic form defining $X$ can be written  
$$
x_1Q_1(x_1,\ldots, x_5)+x_2Q_2(x_1,\ldots, x_5),
$$
for appropriate quadratic forms $Q_1,Q_2$ defined over $\ZZ$. 
With this notation one sees that $Y$ is the locus of
solutions to the system of equations 
$$
x_1=x_2=Q_1(0,0,x_3,x_4,x_5)=Q_2(0,0,x_3,x_4,x_5)=0,
$$
in $\PP^4$. It is now clear that 
the component $Y_0$ of $Y$ of maximal
dimension cannot be a cubic plane curve of the two remaining types.

It remains to deal with the case in which the singular locus $Y$ of $X$
is finite and globally defined over $\QQ$. As shown by C. Segre
\cite{segre}, we have $\delta=\#Y\leq 10$, the extremal
case of $10$ singular points being achieved by the so-called
Segre threefold.  Since $X$ is assumed not to be a cone so we may
assume that all of the singularities are double points. Indeed any
singularity with multiplicity exceeding $2$ must be a vertex for
$X$. In fact, when $\delta\geq 6$ it is known \cite[Lemma 2.2]{6} that
all the singularities are actually nodal.  
Appealing to Coray's partial
resolution of the Cassels--Swinnerton-Dyer conjecture for threefolds,
we are now ready to record our analogue of Lemmas \ref{lem:curve} and \ref{lem:surface}.

\begin{lemma}\label{lem:3fold}
Let $n=5$ and suppose that $X(\QQ)= \emptyset$. Then 
\begin{enumerate}
\item[(i)]
$X$ is non-singular; or
\item[(ii)]
$X$ is a geometrically integral cubic
hypersurface whose singular locus contains precisely $\delta$ double
points, with 
$\delta\in \{3,6, 9\}$.
\end{enumerate}
\end{lemma}

In the second case of Lemma \ref{lem:3fold} it follows from
\cite{c-t-s} and \cite{6} that the Hasse principle holds for $X$ when
$\delta=3$ or $6$. Our investigation would be made easier if we were
also in possession of this fact when $\delta=9$.
Lacking this, all that we actually require from part (ii) of Lemma
\ref{lem:3fold} is that the singular locus should be finite. 
In his survey of open problems in Diophantine geometry, Lewis \cite{lewis}
reports on unpublished joint work with Blass, which would appear to give
Lemma \ref{lem:3fold}. However, in the absence of subsequent
elucidation, 
we have chosen to present our own proof of this result.

\section{Cubic exponential sums}\label{sec:exp_sums}

Let $C\in \ZZ[x_1,\ldots,x_n]$ be an arbitrary
cubic form, assumed to take the shape \eqref{eq:C}. 
Our work in this section centres upon various properties of the cubic
exponential sums
\begin{equation}\label{eq:Tcubic}
S(\al)=S_w(\al;C,P):=\sum_{\x\in\ZZ^n}w(P^{-1}\x)e(\al C(\x)),
\end{equation}
for a suitable family of weights $w$ on $\RR^n$, and cubic forms 
that are always either good (in the sense of the previous section) or the 
hypersurface they define has finite (possibly empty) singular locus.
Specifically, we will 
 collect together some general upper bounds for $S(\al)$, some estimates for suitable moments of $S(\al)$
and some asymptotic formulae for $S(\al)$ when suitable assumptions
are made about how $\al$ can be approximated by rational numbers.
All of these estimates will depend on the parameter $P$ which should
be thought of as tending to infinity.

We must begin by saying a few words about the weight functions that we
will be working with. 
Let $n_1,n_2\geq 0$ such that $n_1+n_2=n$. When $n_i\geq 1$ we let $\z_i\in\RR^{n_i}$
be certain vectors, which we think of as being fixed, but whose
nature will be determined later. Similarly we let $\rho>0$. 
All of the estimates in our work will be allowed to depend upon the
choice of $\z_1,\z_2$ and $\rho$. Define $w_1:\RR^{n_1}\rightarrow
\RR_{\geq 0}$, via 
\begin{equation}
  \label{eq:w1}
w_1(\x_1):=\exp(-\|\x_1-\z_1\|^2(\log P)^4),
\end{equation}
where we have written $\x_1=(x_1,\ldots,x_{n_1})$.
Let $P_0=P(\log P)^{-2}$. 
Then 
$$
w_1(P^{-1}\x_1)=
\exp(-\|\x_1-P\z_1\|^2 P_0^{-2})
$$
is exactly the weight function introduced by Heath-Brown
in \cite[\S 3]{hb-10}.  
Note that 
$$
\nabla w_1(\x_1)=-2(\log P)^4 w_1(\x_1)(x_1-z_1,\ldots,x_{n_1}-z_{n_1}),
$$
so that  $\nabla w_1(\x_1)\ll (\log P)^4$ for
any $\x_1\in\RR^{n_1}$.
Next we let $w_2:\RR^{n_2}\rightarrow \{0,1\}$ denote the 
characteristic function 
\begin{equation}
  \label{eq:w2}
  w_2(\x_2):=
\begin{cases}
1, &\mbox{if $|\x_2-\z_2|<\rho$,}\\
0, &\mbox{otherwise,}
\end{cases}
\end{equation}
where
$\x_2=(x_{n_1+1},\ldots,x_{n})$.

Each weight $w$ appearing in our work will either be of the
form $w_1$ or $w_2$ or $w=(w_1,w_2)$, depending on context. To help
distinguish which estimates are valid for which choice of weight
function, let us denote by $\WW_{n}^{(1)}$ the set of non-negative weight functions on
$\RR^n$  that are of the shape \eqref{eq:w1}, and let 
$\WW_n^{(2)}$ denote the corresponding set of weight functions on
$\RR^n$  of the type \eqref{eq:w2}. We let $\WW_n$ denote the set of
mixed functions $w=(w_1,w_2)$, with $w_i\in \WW_{n_i}^{(i)}$ for $i=1,2$.
In particular $\WW_{n}^{(i)}\subset \WW_n$ for $i=1,2$.
In the definition of these sets the precise value of $\z_1,\z_2$ or $\rho$
is immaterial, unless explicitly indicated otherwise, and the
corresponding implied constants will always be allowed to depend on these
quantities in any way. 

We are now ready to record the upper bounds for $S(\al)$ that feature
in our investigation. 

\begin{lemma}\label{lem:B}
Let $\ve>0$, let $w\in \WW_n$ and assume that $C\in \ZZ[x_1,\ldots,x_n]$ is a good cubic form.
Let $a,q\in \ZZ$ such that $0\leq a< q\leq P^{3/2}$ and
$\gcd(a,q)=1$. Then if $\al=a/q+\theta$ we have
$$
S(\al)\ll P^{n+\ve}\big( q|\theta| +(q|\theta|P^3)^{-1}\big)^{\frac{n}{8}}.
$$
If furthermore $|\theta|\leq q^{-1}P^{-3/2}$, then we have 
$$
S(\al)\ll P^{n+\ve}q^{-\frac{n}{8}}\min\{1,(|\theta|P^3)^{-\frac{n}{8}}\}.
$$
\end{lemma}

\begin{proof}
This is the essential content of the investigation of Davenport
\cite{dav-16} into cubic forms in $16$ variables. 
The bounds are derived in a more succinct manner in
Heath-Brown \cite[\S 2]{14}. The fact that we are
working with exponential sums that are differently weighted
makes no difference to the validity of the argument, and the
reader may wish to consult \cite[\S 9]{41}, where the necessary
modifications can be found in the setting of quartic forms. 
\end{proof}

Define the complete exponential sum 
\begin{equation}
  \label{eq:complete}
S_{a,q}:=\sum_{\y \bmod{q}}e_q(aC(\y)),
\end{equation}
for any coprime integers $a,q$ such that $q>0$. It can easily  be
deduced from the proof of Lemma \ref{lem:B} that $S_{a,q}\ll
q^{7n/8+\ve}$ for any $\ve>0$, under the assumption that the cubic
form is good. The following improvement is due to Heath-Brown \cite[\S 7]{14}.

\begin{lemma}\label{lem:paris}
Let $\ve>0$  and assume that $C\in \ZZ[x_1,\ldots,x_n]$ is a good
cubic form. Then we have $S_{a,q}\ll q^{5n/6+\ve}$.
\end{lemma}

We now come to the real workhorse 
in our argument.  Given $R,\phi>0$ and $v>0$ we define
\begin{equation}\label{eq:sig}
\Sigma_v(R,\phi,\pm):=  
\sum_{R<q\leq 2R} \sum_{\substack{a\bmod{q}\\
    \gcd(a,q)=1}}\int_{\phi}^{2\phi} \big|S\big(\frac{a}{q}\pm t\big)\big|^v \d t.
\end{equation}
The following result provides an upper bound for this quantity.

\begin{lemma}\label{lem:D}
Let $\ve>0$, let $w\in \WW_n$ 
 and assume that $C\in \ZZ[x_1,\ldots,x_n]$ is a good cubic form.
Let $R,\phi>0$, with $R\leq P^{3/2}$ and $\phi\leq R^{-2}$.
Then for any $v\in [0,2]$ and any $H\in
[1,P]\cap \ZZ$ we have
\begin{align*}
\Sigma_v(R,\phi,\pm)\ll P^3
+R^2\phi^{1-\frac{v}{2}} \left(\frac{\psi_H P^{2n-1+\ve}}{H^{n-1}}F\right)^{\frac{v}{2}},
\end{align*}
where
$$
\psi_H:=\phi+\frac{1}{P^2H},
\quad
F:=1+(RH^3\psi_H)^{\frac{n}{2}}+\frac{H^n}{R^{\frac{n}{2}}(P^2\psi_H)^{\frac{n-2}{2}}}.
$$
\end{lemma}

\begin{proof}
It is clear that $\Sigma_0(R,\phi,\pm)\ll R^2\phi.$
Hence it follows from H\"older's inequality that
\begin{align*}
\Sigma_v(R,\phi,\pm)&\ll
(R^2\phi)^{1-\frac{v}{2}} \Sigma_2(R,\phi,\pm)^{\frac{v}{2}}.
\end{align*}
On employing Heath-Brown's estimate for
$\Sigma_2(R,\phi,\pm)$, 
which follows from \cite[Eqs. (4.5) and (5.1)]{14},
we therefore deduce that
\begin{align*}
\Sigma_v(R,\phi,\pm)\ll
(R^2\phi)^{1-\frac{v}{2}} \left(\psi_H R^2\Big( P^2H + \frac{P^{2n-1+\ve}}{H^{n-1}}F\Big)\right)^{\frac{v}{2}}.
\end{align*}
As in the deduction of Lemma \ref{lem:B}, the fact that we are
working with differently weighted exponential sums makes no
difference to the final outcome of the argument.

Using the fact that $R\leq P^{3/2}$ and  $\phi\leq R^{-2}$, with $H\leq P$,
it easily follows that the term involving $P^2H$ contributes
\begin{align*}
\ll 
(R^2\phi)^{1-\frac{v}{2}}(\phi R^2 P^3 +R^2)^{\frac{v}{2}}\ll 
R^2\phi P^{\frac{3v}{2}}+ R^2\phi^{1-\frac{v}{2}}\ll P^3,
\end{align*}
since $0\leq v\leq 2$. This completes the proof of the lemma. 
\end{proof}

Lemma \ref{lem:D} is based on an averaged version of van der Corput's
method and comprises the key innovation in the work of Heath-Brown
\cite{14} already alluded to. Although we have presented it in the
context of denominators $q$ and values of $\al=a/q\pm t$ restricted to
dyadic intervals, the general result consists of a bound for 
$\int |S(\al)|^2 \d\al$, where the integral is taken over a certain
set of minor arcs. For cubic forms in few variables
we will have better results available. When $n=1$ and $w\in \WW_1^{(2)}$, Hua's inequality \cite[Lemma 3.2]{dav-book} implies that
$$
\int_{0}^1|S(\al)|^{2^j}\d\al \ll P^{2^j-j+\ve},
$$
for any $j\leq 3$ . The following result is due to 
Wooley \cite{wooley}, and generalises this to binary forms.

\begin{lemma}\label{lem:hua_2}
Let $\ve>0$, let $w\in\WW_2^{(2)}$ and let 
$C\in \ZZ[x_1,x_2]$ be a binary cubic form, not of the shape
$a(b_1x_1+b_2x_2)^3$, for integers $a,b_1,b_2$. 
Then we have 
$$
\int_{0}^1|S(\al)|^{2^{j-1}}\d\al \ll P^{2^j-j+\ve},
$$
for any $j\leq 3$.
\end{lemma}

Our next selection of results concern the approximation of $S(\al)$ on a certain set of arcs in
the interval $[0,1]$.  For given $A,B,C\geq 0$, define 
$\mathcal{A}=\mathcal{A}(A,B,C)$ to be set of $\al\in [0,1]$ for
which there exists $a,q\in \ZZ$ such that $0\leq a< q\leq P^A$ and
$\gcd(a,q)=1$, with 
\begin{equation}
  \label{eq:Aqa}
\al\in \A_{q,a}:=\Big[\frac{a}{q}-\frac{1}{q^{B}P^{3-C}},  
\frac{a}{q}+\frac{1}{q^{B}P^{3-C}}\Big].
\end{equation}
The major arcs in our work will be a subset of these, but it will be
useful to maintain a certain degree of generality.
When dealing with cubic forms whose singular locus is very small, we
have rather good control over the approximation of $S(\al)$ on the
arcs $\A=\A(A,B,C)$, provided that we work with the
class of smooth weight functions $\WW_n^{(1)}$.  
Recall the definition of $S_{a,q}$ from \eqref{eq:complete} and let
$$
 I_w(\psi):=\int_{\RR^n}w(\x)e(\psi C(\x))\d\x,
$$
for $\psi\in\RR$.
We will need to work with the familiar quantity
\begin{equation}
  \label{eq:S*}
S^*(\al):=q^{-n}P^n S_{a,q}I_w(\theta P^3),
\end{equation}
concerning which we have the following result.

\begin{lemma}\label{lem:F}
Let $\ve>0$ and $n\geq 3$. Assume that
$C\in\ZZ[x_1,\ldots,x_n]$ is a good
cubic form defining a projective hypersurface that is not a cone, with singular locus of
dimension $\sigma\in \{-1,0\}$.
Let $A,B,C\geq 0$ such that $A<1$ and $B\in \{0,1\}$, and let $\al\in
\A_{q,a}$.  Then there exists $w\in\WW_n^{(1)}$  such that 
$$
S(\al)-S^*(\al)\ll P^{A(\frac{n}{3}+\frac{\sigma}{2})+\frac{n+1}{2}+\ve}
+P^{A(1-B)\frac{n+1+\sigma}{2}+C\frac{n+1}{2}+\ve}.
$$
Furthermore, if $kn\geq 12$ and $k\leq 9$, then we have 
$$
\int_\A |S^*(\al)|^k \d\al\ll P^{kn-3+\ve}.
$$
\end{lemma}

\begin{proof}
The proof of this result is based on the investigation carried out by
Heath-Brown \cite{hb-10} into non-singular cubic forms in $10$ variables.
One of the key ingredients in his approach is the Poisson summation
formula, and it is this part of the argument that we plan to take
advantage of. 

We begin by choosing $\z_1\in \RR^{n}$ to be a point at which the
matrix of second derivatives of $C$ has full rank 
at $\z_1$.  The existence of such a point follows from
the work of Hooley \cite[Lemma 26]{hooley2}. 
With this choice of $\z_1$ we now select $w$ to be the weight function
in \eqref{eq:w1}, which belongs to $\WW_n^{(1)}$. Let 
\begin{align*}
S_{a,q}(\v)&:=\sum_{\y \bmod{q}}e_q(aC(\y)+\v.\y),\\
J_w(\psi,\v)&:=\int_{\RR^n}w(P^{-1}\x)e(\psi C(\x)-\v.\x)\d\x,
\end{align*}
for any $\v\in \RR^n$, and let $\al=a/q+\theta \in \A_{q,a}$.
Then \cite[Lemma 8]{hb-10} yields
$$
S(\al)-S^*(\al)\ll 1+ q^{-n}\sum_{\substack{\v\in \ZZ^n\\1\leq |\v|\ll
  V}} S_{a,q}(\v)J_w(\theta,q^{-1}\v),
$$
where
$
V:=(\log P)^7 q (P^{-1}+|\theta|P^2),
$
and furthermore,
\begin{equation}
  \label{eq:bound_99}
J_w(\theta,\mathbf{w})\ll P^n(\log P)^{7n}\min \{1, (|\theta|P^3)^{-1}\}^{\frac{n-1}{2}},
\end{equation}
for any $\mathbf{w}\in \RR^n$.
The main  difference between what we have recorded here and the
statement of  \cite[Lemma 8]{hb-10} is that our definition of $S(\al)$
does not involve a summation over $a$. This
deviation makes no difference to the final outcome.
Note that once the existence of a suitable point $\z_1$ is established
for the definition of the weight function, the manipulations  involving
the exponential integral remain valid even when $C$ is singular. 

The summation over $\v$ in our upper bound for $S(\al)-S^*(\al)$
implies in particular $V\gg 1$. Since $A<1$ we automatically have 
$(\log P)^7 q P^{-1}\leq (\log P)^7P^{A-1}=o(1)$. 
 Hence the condition $V\gg 1$ implies that
$$
(\log P)^7 q |\theta|P^2\leq  V\ll (\log P)^7 q |\theta|P^2
$$
and 
\begin{equation}
  \label{eq:size}
q^{-1}P^{-2}(\log P)^{-7}\ll |\theta| \leq q^{-B}P^{-3+C}.  
\end{equation}
Putting everything together it follows that
$$
S(\al)-S^*(\al)\ll 1+ q^{-n}(\log
P)^{7n}P^{\frac{3-n}{2}}|\theta|^{\frac{1-n}{2}}\mathcal{T}(V),
$$
where
$$
\mathcal{T}(V):=\sum_{1 \leq |\v|\ll V} |S_{a,q}(\v)|.
$$
We will show that
\begin{equation}\label{eq:JB}
\mathcal{T}(V)\ll q^{\frac{n+1+\sigma+\ve}{2}}(V^n+ q^{\frac{n}{3}}),
\end{equation}
for $\sigma\in \{-1,0\}$. 
Before doing so let us see how this suffices to complete the proof of
the first part of the lemma.  Recalling from above that $V$ has
order of magnitude $(\log P)^7q|\theta|P^2$, 
and employing \eqref{eq:size}, we deduce that
\begin{align*}
S(\al)-S^*(\al)
\ll& q^{-\frac{n}{2}+\frac{1+\sigma}{2}+\frac{2\ve}{3}} P^{-\frac{n-3}{2}}|\theta|^{-\frac{n-1}{2}}
((q|\theta|P^2)^{n}+ q^{\frac{n}{3}})\\
\ll& q^{\frac{n+1+\sigma}{2}+\frac{2\ve}{3}}|\theta|^{\frac{n+1}{2}}P^{\frac{3(n+1)}{2}}+
q^{-\frac{n}{6}+\frac{1+\sigma}{2}+\frac{2\ve}{3}}|\theta|^{-\frac{n-1}{2}}P^{-\frac{n-3}{2}}\\
\ll& q^{(1-B)\frac{n}{2}-\frac{B}{2}+\frac{1+\sigma}{2}}P^{C\frac{n+1}{2}+\ve}
+P^{A(\frac{n}{3}+\frac{\sigma}{2})+\frac{n+1}{2}+\ve}.
\end{align*}
If $B=1$ then the first term here is $O(P^{C\frac{n+1}{2}+\ve})$,
since $\sigma\leq 0$.
Alternatively, if $B=0$, then the first term is 
$
O(P^{A\frac{n+1+\sigma}{2}+C\frac{n+1}{2}+\ve}).
$
This establishes the first part of the lemma subject to \eqref{eq:JB}.

To establish \eqref{eq:JB} we return to the manipulations in
\cite[\S 5]{41}. Things are simplified slightly by no longer needing
to keep track of the dependence on $C$ in each implied
constant.  In particular we may take $H\ll 1$ throughout.
The sum $S_{a,q}(\v)$ satisfies a basic multiplicativity property, as
recorded in \cite[Lemma~10]{41}. 
Write $q=bc^2d$, where 
$$
b:=\prod_{\substack{p^e\| q\\ e\leq 2}} p^e, \quad
d:=\prod_{\substack{p^e\| q\\ e\geq 3, 2\nmid e}} p.
$$
In particular $d\mid c$ and we deduce from \cite[Lemmas~7, 10 and 11]{41} that
\begin{align*}
\mathcal{T}(V)
&\ll q^{\frac{n}{2}}b^{\frac{1+\sigma}{2}+\frac{\ve}{2}}\sum_{1 \leq |\v|\ll V}
\sum_{\substack{\mathbf{a}\bmod c\\ c\mid (a\nabla C(\mathbf{a})+\v)}}
N_d(\mathbf{a})^{\frac{1}{2}}.
\end{align*}
Here, if $M(\x)$ denotes the matrix of second derivatives of $C(\x)$,
then $N_m(\mathbf{x})$ is the number of $\y$ modulo $m$ such that
$M(\mathbf{x})\y \equiv \mathbf{0} \bmod{m}$.
Recalling the notation of \cite[Lemma 12]{41},
in which we take $\v_0=\mathbf{0}$ and $g=C$, 
it follows that there is
an absolute constant $\kappa>0$ such that 
\begin{align*}
\mathcal{T}(V)
&\ll  q^{\frac{n}{2}}b^{\frac{1+\sigma+\ve}{2}}\mathcal{S}(\kappa V,a).
\end{align*}
We would now like a version of \cite[Lemma $16$]{41} which
applies to singular forms as well.  We claim that 
\begin{equation}
  \label{eq:JB'} 
\mathcal{S}(\kappa V,a)\ll 
c^\ve d^{\frac{1+\sigma}{2}}V^{n}\Big(1+\frac{c^2d}{V^3}\Big)^{\frac{n}{2}}.
\end{equation}
This relies completely on first establishing suitable analogues of
\cite[Lemmas 13 and 14]{41}.  
A little thought reveals that in the present setting we have
$$
\sum_{|\mathbf{r}|\leq R}N_m(\mathbf{r})^{\frac{1}{2}}\ll
m^{\frac{n}{2}}\Big(1+\frac{R^3}{m}\Big)^{\frac{n}{2}}R^\ve,
$$
for any $m\in  \NN$ and $R\geq 1$.
Here we have used the fact that $C$ is good to bound the number of
$|\mathbf{r}|\leq R$ such that $\rank M(\mathbf{r})=t$, rather than
using \cite[Lemma 2]{41},  as there. Furthermore, we have
$$
\sum_{\mathbf{a}\bmod d}N_d(\mathbf{a}) \ll d^{n+1+\sigma+\ve}.
$$
When $c<V$ it follows from the latter
bound and an application of Cauchy's inequality
that $\mathcal{S}(\kappa V,a)\ll d^{(1+\sigma+\ve)/2}V^n$, which is
acceptable for \eqref{eq:JB'}. In the
alternative case, when $c\geq V$, the necessary modifications to the 
proof of \cite[Lemma $16$]{41} are straightforward and we omit full details here.

We may now insert \eqref{eq:JB'} into the preceeding estimate for
$\mathcal{T}(V)$ to conclude that 
$$
\mathcal{T}(V)
\ll  q^{\frac{n+1+\sigma+\ve}{2}}
V^{n}\Big(1+\frac{q}{V^3}\Big)^{\frac{n}{2}}.
$$
If $q^{1/3}\leq V$ then this is clearly satisfacory for \eqref{eq:JB}. 
Alternatively, if $V<q^{1/3}$ then we can only enlarge our bound for
$\mathcal{T}(V)$ if we replace $V$ by $q^{1/3}$. But then 
$\mathcal{T}(V)$ is easily seen to be bounded by \eqref{eq:JB} in this
case too. This therefore completes the proof of \eqref{eq:JB}.

Our final task is to establish the second part of the lemma. Since $C$ is good, 
we may combine  Lemma \ref{lem:paris} with 
\eqref{eq:bound_99} to deduce that 
$$
S^*(\al)\ll q^{-\frac{n}{6}} P^n(\log P)^{7n} 
\min \{1, (|\theta|P^3)^{-1}\}^{\frac{n-1}{2}}.
$$
Let us write $T=q^{-B}P^{-3+C}$ for convenience. It therefore follows that
\begin{align*}
\int_{\A} |S^*(\al)|^k\d\al 
&\ll P^{kn+\frac{\ve}{2}}\sum_{q\leq P^A}
q^{1-\frac{kn}{6}} \int_{-T}^{T} 
\min \{1, (|\theta|P^3)\}^{-\frac{k(n-1)}{2}}\d\theta\\
&\ll P^{kn-3+\frac{\ve}{2}}\sum_{q\leq P^A}
q^{1-\frac{kn}{6}} \\
&\ll P^{kn-3+\ve},
\end{align*}
since $kn\geq 12$ and $k\leq 9$.
\end{proof}

We remark that when $\sigma=0$ it seems likely that an even sharper
error term is available in Lemma \ref{lem:F} through a more careful
analysis of the complete exponential sums $S_{a,q}(\v)$, when $q$ is prime.
It follows from \cite[Lemma 28]{hooley2} that the form 
$C$ is automatically good when the corresponding hypersurface has at most isolated
singularities and these are suitably mild.

In the setting of $1$-dimensional exponential sums, we have even
better control over $S(\al)$ on the arcs $\A=\A(A,B,C)$.  Let $C(x)=cx^3$ for some non-zero coefficient $c\in \ZZ$.  Then for
any $a,q\in \ZZ$ such that $0\leq a< q\leq P^A$ and $\gcd(a,q)=1$, and
any $\al=a/q+\theta\in \A_{q,a}$, the standard major arc analysis would
provide an estimate of the shape
$$
S(\al)= S^*(\al)+O(P^{A}+P^{A+C-AB}),
$$
where $S^*(\al)$ is given by \eqref{eq:S*}. 
Our final result in this section improves on this substantially,
and is readily derived from the book of Vaughan \cite[\S 4]{vaughan}.

\begin{lemma}\label{lem:E}
Let $\ve>0$, let $n=1$ and let $w\in\WW_1^{(2)}$. 
Let $A,B,C\geq 0$ with $A,B\leq 1$. Then for any $\al\in
\A_{q,a}$ we have 
$$
S(\al)=S^*(\al)+O(P^{\frac{A}{2}+\ve}+P^{\frac{A+C-AB}{2}+\ve}).
$$
Furthermore, if $k\geq 4$, then we have 
$$
\int_\A |S^*(\al)|^k \d\al\ll P^{k-3+\ve}.
$$
\end{lemma}

\section{Density of rational points on cubic hypersurfaces}\label{sec:upper}

Let $X \subset \PP^{n-1}$ be a cubic hypersurface, not equal to a
cone, that is defined by an absolutely irreducible cubic form $F\in
\ZZ[x_1,\ldots,x_n]$. For $P\geq 1$, let
$$
N_{n,F}(P):=\#\{\x\in\ZZ^n: |\x|\leq P, ~F(\x)=0\},
$$
According to the conjecture of Manin \cite{f-m-t} one expects 
$N_{n,F}(P)\sim cP^{n-3}$ for some constant $c\geq 0$ as soon as $F$ is
non-singular and $n\geq 5$. 
When $F$ is not necessarily non-singular, or the number of variables
is small, there is the dimension growth conjecture due to Heath-Brown. This
predicts that
\begin{equation}
  \label{eq:upper_2}
N_{n,F}(P)\ll P^{n-2+\ve},
\end{equation}
and has received a great deal of attention in recent years.
Let $\sigma$ denote the projective dimension of the singular
locus of $X$. The dimension growth  conjecture has been established by the author
\cite{cubhyp-circle} when $n\geq 6+\sigma$.
The following result, which may of independent interest,  
shows that one can do better than 
\eqref{eq:upper_2} if larger values of $n$ are permitted.

\begin{lemma}\label{lem:8}
We have $N_{n,F}(P)\ll P^{n-5/2+\ve}$ when $n\geq 9+\sigma$.
\end{lemma}

\begin{proof}
Our proof of the lemma is based on the approach developed in \cite{cubhyp-circle}.
Arguing with hyperplane sections, as in \cite[\S 2]{cubhyp-circle}, we
see that it will suffice to show that there is an absolute constant
$\theta>0$ such that 
\begin{equation}
  \label{eq:gen1}
N_{w}(g;P)
:= \sum_{\substack{\x\in \ZZ^n\\ g(\x)=0}} w(P^{-1}\x)
\ll H^{\theta}P^{n-\frac{5}{2}+\ve},
\end{equation}
for any weight function $w:\RR^n\rightarrow \RR_{\geq 0}$ belonging to
the class of weight functions described at the start of \cite[\S 2]{cubhyp-circle}, any 
$H \geq \|g\|_P$, and any 
cubic polynomial $g\in  \ZZ[x_1,\ldots,x_n]$ such that $n\geq 8$ and 
the cubic part $g_0$ is non-singular.
Here we recall that $\|g\|_P:= \|P^{-3}g(P\x)\|$, where $\|h\|$
denotes the height of a polynomial $h$. 

The bulk of \cite{cubhyp-circle} goes through
verbatim, and we are left with reevaluating the 
estimation of $\Sigma_2=\Sigma_2(R,\mathbf{R};t)$ and $\Sigma_1=\Sigma_1(R,\mathbf{R};t)$ in 
\cite[\S 5.1]{cubhyp-circle} and \cite[\S 5.2]{cubhyp-circle}, respectively.
Beginning with the former, 
we note from \cite[Eq. (5.5)]{cubhyp-circle} 
that this breaks into an estimation of $\Sigma_{2,a}$ and $\Sigma_{2,b}$.
The first of these is estimated as $O(H^{\theta}
P^{3n/4-3/4+\ve}+H^{\theta} P^{n-3+\ve})$. Both of the exponents of
$P$ are clearly at most $n-5/2+\ve$ when $n\geq 8$, as required for
\eqref{eq:gen1}.
Turning to $\Sigma_{2,b}$, one easily traces through the argument,
finding that
$$
\Sigma_{2,b}\ll H^\theta P^\ve \big(
P^{\frac{3n}{4}-\frac{3}{4}}+ P^{n-2}E_n + P^{\frac{13n}{16}-1}+P^{\frac{7n}{8}-\frac{5}{3}} 
\big),
$$
where 
$$
E_n=P^{-1-\frac{7n}{40}}R^{2-\frac{3n}{20}}R_2^{\frac{3n}{10}-\frac{3}{2}}\ll 
P^{-1-\frac{7n}{40}}R^{\frac{5}{4}}\ll  P^{\frac{7}{8}-\frac{7n}{40}}.
$$
Here the first term (resp. second term, sum of the final two terms) corresponds to the case $V\geq R_2$
(resp. $(R_2^2R_3)^{1/3}\leq V<R_2$, $V<(R_2^2R_3)^{1/3}$). 
A modest pause for thought reveals that all of these exponents are
satisfactory when $n\geq 8$.

We now turn to the estimation of $\Sigma_1$ in 
\cite[\S 5.2]{cubhyp-circle}, which is again written as a sum 
$\Sigma_{1,a}+\Sigma_{1,b}$.  Beginning with $\Sigma_{1,a}$, we easily
observe that
$$
\Sigma_{1,a}\ll H^\theta P^\ve \big(
P^{\frac{3n}{4}-\frac{3}{4}}+ P^{n-2}E_n + P^{n-3}\big),
$$
this time with 
$$
E_n=P^{2-\frac{n}{4}}t^{1-\frac{n}{12}}R^{\frac{11}{6}-\frac{n}{4}}(R_2^2R_3)^{\frac{n}{9}-\frac{1}{2}}
\ll P^{-\frac{1}{2}}R^{-\frac{1}{9}}+P^{-1}R^{\frac{2}{9}}\ll P^{-\frac{1}{2}}.
$$
This therefore shows that $\Sigma_{1,a}\ll H^\theta P^{n-5/2+\ve}$, as
required for \eqref{eq:gen1}.  Turning to 
$\Sigma_{1,b}$,  we will need to modify  the argument slightly.
On noting that $R_2^{3/2}R_3^{1/2}\gg (R_2^2R_3)^{2/3}$,
we easily deduce that 
$$
\Sigma_{1,b}\ll H^\theta P^\ve \big(
P^{\frac{3n}{4}-\frac{3}{4}}+ P^{n-2+\frac{7}{8}-\frac{7n}{40}} + T\big),
$$
where we have set
$$
T:=P^nt R^{2-\frac{n}{2}} (R_2^2 R_3)^{-\frac{2}{3}}
\min\Big\{R_2^n, 
\Big(\frac{R_2^2R_3}{V}\Big)^{\frac{n}{2}}, R^{\frac{3n}{8}}\min\{1, (tP^3)^{-\frac{n}{8}}
\Big\}.
$$
The first and second terms here are satisfactory for $n\geq 8$.
Moreover the third term is clearly satisfactory for $n\geq 16$, on
taking $\min\{A,B,C\}=C$. To handle the contribution from the final term when $8\leq n<16$, it will be convenient
to recall that $V$ has order of magnitude $Rt^{1/2}P^{1/2}$ when
$t\geq P^{-3}$ and $R/P$ when $t<P^{-3}$. 

Suppose first that $R\geq
P$. When $t\geq P^{-3}$ we deduce that 
\begin{align*}
T\ll P^{n-3} \min\big\{ 
R^{2-n}(R_2^2R_3)^{\frac{n}{2}-\frac{2}{3}}P^{\frac{n}{2}},
R^{2-\frac{n}{8}} (R_2^2R_3)^{-\frac{2}{3}}\big\}
\ll P^{n-\frac{7}{3}}R^{\frac{5}{6}-\frac{n}{8}}
&\ll P^{\frac{7n}{8}-\frac{3}{2}},
\end{align*}
on taking 
\begin{equation}
  \label{eq:ab}
  \min\{A,B\}\leq A^{\frac{4}{3n}}B^{1-\frac{4}{3n}}.
\end{equation}
This is clearly satisfactory for $n\geq 8$. When $t<P^{-3}$ we easily deduce that the
same bound holds on taking $V$ to be of size $R/P$ in the definition
of $T$.

Suppose now that $R<P$ and $t\geq P^{-3}$. 
If $t>(R^2P)^{-1}$, then it is not hard to see that
\begin{align*}
T\ll P^{n-1} R^{-\frac{n}{2}} (R_2^2 R_3)^{-\frac{2}{3}}
\min\big\{(R_2^2R_3)^\frac{n}{2}, R^{\frac{5n}{8}}P^{-\frac{n}{4}}\big\}
\ll P^{\frac{3n}{4}-\frac{2}{3}}R^{\frac{n}{8}-\frac{5}{6}}
&\ll P^{\frac{7n}{8}-\frac{3}{2}},
\end{align*}
using \eqref{eq:ab}.  This is satisfactory for $n\geq
8$. Alternatively, if $P^{-3}\leq 
t\leq (R^2P)^{-1}$, then one finds that
\begin{align*}
T&\ll P^{n} R^{2-\frac{n}{2}} (R_2^2 R_3)^{-\frac{2}{3}}
\min\big\{t(R_2^2R_3)^\frac{n}{2},
R^{\frac{3n}{8}}t^{1-\frac{n}{8}}P^{-\frac{3n}{8}}\big\}.
\end{align*}
Using \eqref{eq:ab} it easily follows that
\begin{align*}
T&\ll P^{\frac{5n}{8}+\frac{1}{2}} t^{\frac{7}{6}-\frac{n}{8}}R^{\frac{3}{2}-\frac{n}{8}}.
\end{align*}
Since $t\leq (R^2P)^{-1}$ and $R<P$ this is clearly satisfactory when
$n=8$.  If instead $n\geq 9$ then we deduce that
\begin{align*}
T&\ll P^{n-3} R^{5-\frac{n}{2}}(R_2^2 R_3)^{\frac{n}{2}-\frac{14}{3}},
\end{align*}
on taking
$\min\{A,B\}\leq A^{1-\frac{8}{n}}B^{\frac{8}{n}}$ rather than \eqref{eq:ab}.
Since $R<P$ we easily conclude that 
$T\ll P^{n-5/2}$ in this case too.
Finally, when $R<P$ and $t<P^{-3}$, we see that
\begin{align*}
T\ll P^{n-3}
\min\big\{ R^{2-\frac{n}{2}}(R_2^2R_3)^{\frac{n}{2}-\frac{2}{3}}, R^{2-\frac{n}{8}}(R_2^2R_3)^{-\frac{2}{3}} 
\big\} \ll P^{n-3} R^{\frac{3}{2}-\frac{n}{8}}
&\ll P^{n-\frac{5}{2}},
\end{align*}
using \eqref{eq:ab}. This therefore concludes the proof of the lemma.
\end{proof}

It is clear from the proof of Lemma \ref{lem:8} that one actually
achieves an estimate of the shape 
$$
N_{n,F}(P)\leq c_{\ve,n}\|F\|^\theta P^{n-5/2+\ve},
$$ 
for a constant
$\theta>0$, when $n\geq 9+\sigma$.
It seems likely that one can push the analysis further, obtaining
$N_{n,F}(P)\ll P^{n-3+\ve}$ for $n\geq 11+\sigma$, as predicted by Manin.

A key step in our argument involves generating good estimates for the
moments 
\begin{equation}
  \label{eq:moment}
  M_{n}(P):=\int_{0}^1|S(\al)|^{2}\d\al,
\end{equation}
where $S(\al)$ is the cubic exponential sum \eqref{eq:Tcubic}, for
an appropriate weight $w\in\WW_n$. 
By the orthogonality of the exponential
function we have 
$$
M_n(P) =\sum_{\substack{\x,\y\in\ZZ^n\\C(\x)=C(\y)}} w(P^{-1}\x)w(P^{-1}\y).
$$
It is clear that there exists a constant $c>0$ depending on $w$ such
that the overall contribution to $M_n(P)$
from $\x,\y$ such that $\max\{|\x|, |\y|\} >c P$ is
$O(1)$,  if $P$ is taken to be sufficiently large. Hence
it follows that
$$
M_n(P)\ll N_{2n,C-C}(cP).
$$
When $C$ is a non-singular form in $n$ variables it is obvious that $C-C$
is a non-singular form in $2n$ variables, defining a hypersurface of
dimension $2n-2$.  
When $C$ has a finite non-empty singular locus it is not hard to see
that $C-C$ has singular locus of dimension $1$.
The following result now flows very easily from
\eqref{eq:upper_2} and Lemma \ref{lem:8}.

\begin{lemma}\label{lem:C}
Let $\ve>0$ and let $n\geq 3$.
Assume that $C\in \ZZ[x_1,\ldots,x_n]$ is a cubic form defining a
projective hypersurface whose singular locus has dimension $\sigma\in\{-1,0\}$.
Then we have  
$$
M_n(P) \ll 
\begin{cases}
P^{4+\ve},
&\mbox{if $n=3$ and $\sigma=-1$,}\\
P^{2n-\frac{5}{2}+\ve}, &\mbox{if $n\geq 5+\sigma$.}
\end{cases}
$$
\end{lemma}

\section{Cubics splitting off a form}

In this section we establish Theorem \ref{main}. 
Let $n_1,n_2\geq 1$ such that 
$$
n_1+n_2=n\geq 13.
$$ 
It will be convenient to write $\x=(x_1,\ldots,x_{n_1})$ and $\y=(y_1,\ldots,y_{n_2})$.
We henceforth fix our attention on cubic forms of the shape 
$$
C(\x,\y)=C_1(\x)+C_2(\y),
$$
with $C_1\in \ZZ[\x]$ and $C_2\in \ZZ[\y]$.  
In what follows we may always suppose that $C=C_1+C_2$ is non-degenerate, by which we mean that it
is not equivalent over $\ZZ$ to a cubic form in fewer variables, since
such forms have obvious non-zero integral solutions.
Recall the definition of ``good'' cubic forms from \S \ref{s:geom}.
It follows from \cite[\S 14]{dav-book} that either $C_1$ is good or else the cubic hypersurface $C_1=0$ has a
rational point. The same is true for the cubic forms $C_2$ and
$C_1+C_2$. Since the existence of a rational point on any of these
hypersurfaces is enough to ensure that $X(\QQ)\neq \emptyset$ in the
statement of Theorem \ref{main}, so we may proceed under the assumption that
$C_1, C_2$ and $C_1+C_2$ are all good.

Let $w=(w_1,w_2)\in \WW_n$, as introduced in \S \ref{sec:exp_sums}.
When $C_1$ satisfies the hypotheses in Lemma \ref{lem:F} we will
assume that $w_1\in\WW_{n_1}^{(1)}$ is the weight function
constructed there.  
Our argument revolves around establishing an asymptotic formula for
the sum
$$
N(P):=\sum_{\substack{(\x,\y)\in \ZZ^n\\ C_1(\x)+C_2(\y)=0}} 
w_1(P^{-1}\x)w_2(P^{-1}\y),  
$$
as $P\rightarrow \infty$.  
As is usual in applications of the Hardy--Littlewood circle method,
the starting point is the simple identity
$$
N(P)=\int_0^1 S_1(\al)S_2(\al) \d\al,  
$$
where 
$$
S_i(\al):=\sum_{\x\in\ZZ^{n_i}}w_i(P^{-1}\x) e(\al C_i(\x))  
$$
for $i=1,2$. It will be convenient to define 
$$
S(\al):=S_1(\al)S_2(\al)=\sum_{(\x,\y)\in\ZZ^{n}}w(P^{-1}(\x,\y))
e\big(\al (C_1(\x)+C_2(\y))\big),
$$
where $w=(w_1,w_2)$.  Then $S(\al)=S_w(\al;C_1+C_2,P)$
is a cubic exponential sum of the sort introduced in \eqref{eq:Tcubic}.

In the usual way one divides the interval $[0,1]$ into a
set of major arcs and minor arcs.  For major arcs we  will take 
the union of intervals 
$$
\M:=\bigcup_{q\leq P^{\D}}\bigcup_{\substack{a=0\\ \gcd(a,q)=1}}^{q-1} \Big[\frac{a}{q}-
P^{-3+\D}, \frac{a}{q}+P^{-3+\D}
\Big],
$$
which is equal to $\A(\D,0,\D)$ in the notation of \eqref{eq:Aqa}.
 The corresponding set of minor arcs is defined modulo $1$ as
$\m=[0,1]\setminus \M$.
Here $\D>0$ is an arbitrary small parameter. It turns out the choice 
$$
\D:=\frac{1}{10}
$$
is acceptable.  
We may deduce from Lemma~15.4 and \S\S 16--18 in \cite{dav-book} 
(see also \cite[Lemma~2.1]{14}) that
$$
\int_\M S(\al)\d\al = \ss \ii P^{n-3}+o(P^{n-3}),
$$
where 
\begin{align*}
\ss&:=\sum_{q=1}^\infty 
\sum_{\substack{a \bmod{q}\\ \gcd(a,q)=1}} 
q^{-n}S_{a,q}^{(1)}S_{a,q}^{(2)},\\
\ii&:=\int_{-\infty}^{\infty}
\int_{\RR^{n_1}}\int_{\RR^{n_2}}w(\x,\y)e\big(\theta(C_1(\x)+C_2(\x))\big)\d\x\d\y
\d \theta
\end{align*}
are both absolutely convergent.
Here, the absolute convergence of $\ss$ follows from Lemma
\ref{lem:paris}, and  we have written
$$
S_{a,q}^{(i)}:=\sum_{\mathbf{u}\in (\ZZ/q\ZZ)^{n_i}} e_q\big(aC_i(\mathbf{u})\big),
$$
for $i=1,2$.
Since $\ss$ is absolutely convergent and $C=C_1+C_2$ is
non-degenerate, it follows from standard arguments (see \cite[Lemma 7.3]{dav-32}, for example)
that $\ss>0$. The treatment of the singular integral is routine and we omit giving the details
here, all of which can be supplied by consulting 
\cite[\S 16]{dav-book} and \cite[\S 4]{hb-10}. Assuming that 
neither $C_1$ nor $C_2$ has a linear factor defined over
$\QQ$ it is possible to choose 
$(\z_1,\z_2)\in \RR^{n}$ in the definition of
$w=(w_1,w_2)$, so that 
each $\z_i$ is a non-singular real solution to $C_i=0$. On 
selecting a sufficiently small value of $\rho>0$ in the definition
of $w_2$ we can then ensure $\ii> 0$. 
The case in which $C_1$ or $C_2$ does factorise over $\QQ$ 
clearly enables us to deduce the statement of Theorem \ref{main}
very easily.

In order to conclude the proof of
Theorem~\ref{main} it remains to show that the overall contribution
from the minor arcs 
\begin{equation}
  \label{eq:minor}
E:=\int_{\m} S_1(\al)S_2(\al) \d \al,
\end{equation}
is satisfactory.  This is
where the bulk of our work lies and we will find it necessary to
undertake a lengthy case by case analysis to handle the 
different values of $n_1$ and
$n_2$. In doing so it will suffice to handle the case $n_1+n_2=13$,
the case $n_1+n_2>13$ being taken care of by \cite{14}.
Without loss of generality we assume henceforth that $1\leq n_1\leq 6$.

Let $Q\geq 1$ and let $\al \in \m$. By Dirichlet's approximation theorem
we may find coprime integers $1 \leq a \leq q$ such that $q \leq Q$
and $|q\al-a|\leq 1/Q$.   The value of $Q$ should satisfy $1\leq Q
\leq P^{3/2}$ and is chosen to optimise the final stages of the argument.
The obvious approach involves applying estimates for each individual
exponential sum $S_1(\al)$ and $S_2(\al)$ for $\al\in \m$, before then
deriving an estimate for the integral over the full set of minor
arcs. While we have rather good control over these sums when $n_1$ and
$n_2$ are both large,
the case in which one of $n_1$ or $n_2$ is small presents more of an
obstacle.  Instead we apply H\"older's
inequality to deduce that
\begin{equation}
  \label{eq:split}
  |E| \leq \left(\int_{\m} |S_1(\al)|^u \d \al \right)^{\frac{1}{u}}\left(\int_{\m} |S_2(\al)|^v \d \al \right)^{\frac{1}{v}},
\end{equation}
for any $u,v>0$ such that $1/u+1/v=1$. This will allow us to separate
out the behaviour of the exponential 
sums $S_1(\al)$ and $S_2(\al)$ on the minor arcs.

Before embarking on the case by case analysis alluded to above, it
will save needless repetition if we give some reasonably general
estimates here that can be applied in various contexts.
Our principal means for dealing with small values of $n_1$ relies on
taking the inequality 
\begin{equation}
  \label{eq:Ivu}
\int_{\m} |S_1(\al)|^u \d \al \leq \int_{0}^1 |S_1(\al)|^u \d \al
\end{equation}
in \eqref{eq:split}.  This will in turn be estimated as $O(P^{k+\ve})$
for an appropriate $k>0$, whence a typical scenario 
entails studying
\begin{equation}
  \label{eq:Iuv}
  I_{u,v}(k;\mathfrak{n}):=P^{\frac{k}{u}+\ve}\left(\int_{\mathfrak{n}}|S_2(\al)|^v \d\al\right)^{\frac{1}{v}},
\end{equation}
for $u,v>0$ such that $1/u+1/v=1$ and certain subsets
$\mathfrak{n}\subseteq\m$. We will always assume that $6/5< v\leq 2$.

Let $\al \in \mathfrak{n}$ and let $Q\geq 1$.  There exist
coprime integers $0 \leq a <q \leq Q$ such that 
 $|q\al-a|\leq 1/Q$.  An argument based on dyadic summation reveals that 
\begin{equation}
  \label{eq:foot}
  I_{u,v}(k;\mathfrak{n})\ll P^{\frac{k}{u}+\ve} (\log P)^2 \max_{R,\phi,\pm} \Sigma_v(R,\phi,\pm)^{\frac{1}{v}}, 
\end{equation}
where $\Sigma_v(R,\phi,\pm)$ is given by \eqref{eq:sig}, and 
the maximum is over the possible sign changes and 
$R,\phi$ such that 
\begin{equation}
  \label{eq:range}
0<R\leq Q, \quad 0<\phi\leq (RQ)^{-1}.
\end{equation}
Furthermore, $R,\phi$ should satisfy whatever conditions are
appropriate to ensure we are dealing with points on $\mathfrak{n}$. In
particular, since $\mathfrak{n}\subseteq \m$  the inequalities
$R\leq P^\D$ and $\phi \leq P^{-3+\D}$ cannot both
hold simultaneously. 

Let $u,v,k$ be given. Define
\begin{equation}
  \label{eq:rp1}
  \rho_n:=\frac{n(vn-8-v)}{vn^2-(3v+4)n+2v}, \quad \pi_n:=
\frac{-2v(n^2-(18-2\frac{k}{u})n-2)}{vn^2-(3v+4)n+2v},
\end{equation}
and 
\begin{equation}
  \label{eq:rp2}
  \rho_n':=\frac{2v}{2-v}\left(\frac{n^2}{2(3n-2)}-\frac{2}{v}\right), \quad \pi_n':=
\frac{2v}{2-v}\left(n-\frac{23}{2}+\frac{k}{u}\right).
\end{equation}
Let
\begin{equation}
  \label{eq:delta}
\delta:=\frac{1}{10^4}.
\end{equation}
Recall that our task is to show that
$E=o(P^{10})$ when $n=n_1+n_2=13$.
The following result provides us with easily checked
conditions on $u,v,k$ and $n_2$ under which 
$I_{u,v}(k;\mathfrak{n})$ makes a satisfactory contribution.

\begin{lemma}\label{lem:v<2}
Let $6/5< v<2$. Assume that $n_2\geq 6$ and 
\begin{equation}
  \label{eq:paris}
  \rho_{n_2}+\rho_{n_2'}\geq 1.
\end{equation}
Define $\m_0$ to be the set 
of $\al \in \m$ for which there exist coprime integers $0 \leq a <q$ such that 
$$
q\leq P^{\frac{\pi_{n_2}'-\pi_{n_2}}{\rho_{n_2}'+\rho_{n_2}}+2\delta}, \quad
q^{\rho_{n_2}'}P^{-\pi_{n_2}'-\delta}\leq 
\Big|\al-\frac{a}{q}\Big|\leq q^{-\rho_{n_2}}P^{-\pi_{n_2}+\delta}.
$$
Then 
$$
I_{u,v}(k;\mathfrak{n}\setminus \m_0)=o(P^{10}),
$$
for any $\mathfrak{n}\subseteq \m$,
provided that 
\begin{equation}\label{eq:in1}
\frac{2}{v}+\frac{21}{2}-n_2 \leq \frac{k}{u}< \frac{103}{10}-\frac{4n_2}{5}
\end{equation}
and 
\begin{equation}
  \label{eq:in2}
\pi_{n_2}'\geq 3.
\end{equation}
\end{lemma}

\begin{proof}
We will commence under the assumption that $6/5<v\leq 2$,
saving the restriction $v<2$ until later in the argument.
It is clear that $I_{u,v}(k;\mathfrak{n}\setminus \m_0)\leq I_{u,v}(k;\m\setminus \m_0)$.
Let us consider the consequences of applying Lemma \ref{lem:D} in our
estimate \eqref{eq:foot} for $I_{u,v}(k;\m\setminus \m_0)$. Throughout
the proof of Lemma \ref{lem:v<2} we will denote $\m\setminus \m_0$ by
$\mathfrak{a}$ and we will set $n=n_2$. We may deduce from Lemma
\ref{lem:D} and \eqref{eq:Iuv} that
$$
I_{u,v}(k;\mathfrak{a})
\ll P^{\frac{k}{u}+\ve} \left(P^{\frac{3}{v}}+\max_{R,\phi}
R^{\frac{2}{v}}\phi^{\frac{1}{v}-\frac{1}{2}} \Big( \frac{\psi_H
  P^{2n-1}}{H^{n-1}}F\Big)^{\frac{1}{2}} \right),
$$
where $\psi_H$ and $F$ are as in the statement of the lemma
and $H\in [1,P]\cap \ZZ$ is arbitrary. Furthermore the maximum is over
$R,\phi$ such that
\eqref{eq:range} holds with any choice of $Q\geq 1$ that we care
to choose. We write $Q=P^{\kappa}$, with 
\begin{equation}
  \label{eq:kappa}
\kappa:=\frac{3(2n-21+2\frac{k}{u})}{n-1}+3\ve.
\end{equation}
In particular one easily checks that $0\leq \kappa\leq
3/2$ if \eqref{eq:in1}  holds and $\ve>0$ is sufficiently small.
It follows from \eqref{eq:in1} that $k/u<11/2$ since $n\geq 6$. Hence
the term involving $P^{3/v}$ contributes $O(P^{8+\ve})$, which is satisfactory.

Let  us now turn to the contribution from the term involving
$F$ in our estimate for $I_{u,v}(k;\mathfrak{a})$. Define
\begin{equation}
  \label{eq:sock}
  \phi_0:=(R^{-\frac{2}{v}}P^{-(2n-\frac{23}{2}+\frac{k}{u})})^{\frac{2v}{v(n-1)+2}}.
\end{equation}
Then our investigation will be optimised by taking
$$
H:=
\begin{cases}
\lfloor P^{\ve}\max\{1,R^{\frac{2}{v}}\phi^{\frac{1}{v}}P^{n-\frac{21}{2}+\frac{k}{u}}\}^{\frac{2}{n-1}} \rfloor, &\mbox{if  $\phi>\phi_0$,}\\
\lfloor P^{\ve}\max\{1,R^{\frac{2}{v}}\phi^{\frac{1}{v}-\frac{1}{2}}P^{n-\frac{23}{2}+\frac{k}{u}}\}^{\frac{2}{n}}\rfloor, &\mbox{if  $\phi\leq \phi_0$.}
\end{cases}
$$
If we can show that $F\ll 1$ with this
choice of $H$ then we will have 
$$
\frac{P^{n-\frac{1}{2}+\frac{k}{u}+\ve}
  R^{\frac{2}{v}}\phi^{\frac{1}{v}-\frac{1}{2}}\psi_H^{\frac{1}{2}}}{H^{\frac{n-1}{2}}}
F^{\frac{1}{2}}\ll P^{10-\frac{3\ve}{2}},
$$
since $n \geq 6$. We deduce from \eqref{eq:range} that 
\begin{align*}
(R^{\frac{2}{v}}\phi^{\frac{1}{v}}P^{n-\frac{21}{2}+\frac{k}{u}})^{\frac{2}{n-1}}
&\leq P^{\frac{2(n-\frac{21}{2}+\frac{k}{u})}{n-1}},
\quad (R^{\frac{2}{v}}\phi^{\frac{1}{v}-\frac{1}{2}}P^{n-\frac{23}{2}+\frac{k}{u}})^{\frac{2}{n}}
\leq P^{\frac{2(n-10+\frac{k}{u})}{n}}.
\end{align*}
In either case the final exponent of $P$ is less than 1, by
\eqref{eq:in1}.  
Hence $H$ is an integer in the
interval $[1,P]$ and it remains to show that $F\ll 1$ with this choice of
$H$. Recall the definition of $F$ from Lemma \ref{lem:D}.

Suppose first that $\phi>\phi_0$, with $\phi_0$ given by
\eqref{eq:sock}. Then $\psi_H\ll \phi$ and it follows that
\begin{align*}
RH^3\psi_H
\ll R\phi P^{3\ve}
(1+R^{\frac{2}{v}}\phi^{\frac{1}{v}}P^{n-\frac{21}{2}+\frac{k}{u}})^{\frac{6}{n-1}}
&\ll 1+ P^{-\kappa} P^{(n-\frac{21}{2}+\frac{k}{u})\frac{6}{n-1}+3\ve},
\end{align*}
by \eqref{eq:range}. It follows from our expression \eqref{eq:kappa}
for $\kappa$ that this is $O(1)$.
Turning to the third term in the definition of 
$F$, we see that
\begin{align*}
\frac{H^{n}}{R^{\frac{n}{2}}(P^2\psi_H)^{\frac{n-2}{2}}}\ll 
\frac{P^{n\ve}}{R^{\frac{n}{2}}\phi^{\frac{n-2}{2}}P^{n-2}}(1+R^{\frac{2}{v}}\phi^{\frac{1}{v}}
P^{n-\frac{21}{2}+\frac{k}{u}})^{\frac{2n}{n-1}}.
\end{align*}
The exponent of $\phi$ in the second term is $2n/(v(n-1))-(n-2)/2$, which is negative
since $v> 6/5$ and $n\geq 6$. Hence this
quantity is $O(1)$ provided that $\phi\geq \phi_1$, with 
\begin{equation}
  \label{eq:knee}
  \phi_1:=R^{-\rho_n}P^{-\pi_n+\delta},
\end{equation}
with $\rho_n, \pi_n$ given by \eqref{eq:rp1} and $\delta$ given by \eqref{eq:delta}.
Here, as is customary, we have assumed that $\ve$ is
sufficiently small.
One also checks that taking $\phi\geq \phi_1$ is enough to ensure
that the first term is $O(1)$, in view of the lower bound for $k/u$ in \eqref{eq:in1}.

Suppose now that $\phi\leq \phi_0$.  Then $\psi_H\ll (P^2H)^{-1}$ and it follows that
\begin{align*}
RH^3\psi_H
&\ll \frac{R P^{2\ve}}{P^2}
(1+R^{\frac{2}{v}}\phi^{\frac{1}{v}-\frac{1}{2}}P^{n-\frac{23}{2}+\frac{k}{u}})^{\frac{4}{n}}\\
&\ll 1+  
R^{1+\frac{8}{vn}}(RQ)^{-(\frac{1}{v}-\frac{1}{2})\frac{4}{n}}P^{(n-\frac{23}{2}+\frac{k}{u})\frac{4}{n}-2+2\ve}\\
&\ll 1+  
P^{\kappa(1+\frac{4}{n})}P^{2-\frac{46}{n}+\frac{4k}{un}+2\ve},
\end{align*}
since $Q=P^{\kappa}$. It follows from \eqref{eq:in1} and \eqref{eq:kappa} 
that this is $O(1)$.  Thus the second term makes a satisfactory 
contribution in $F$. Turning to the third term, we find that
\begin{align*}
\frac{H^{n}}{R^{\frac{n}{2}}(P^2\psi_H)^{\frac{n-2}{2}}}\ll
 \frac{P^{\frac{(3n-2)\ve}{2}}}{R^{\frac{n}{2}}}
(1+R^{\frac{2}{v}}\phi^{\frac{1}{v}-\frac{1}{2}}P^{n-\frac{23}{2}+\frac{k}{u}})^{\frac{3n-2}{n}}.
\end{align*}
We now make the assumption $6/5<
v<2$. Hence the overall contribution from the second term
is $O(1)$ provided that $\phi\leq \phi_2$, with 
\begin{equation}
  \label{eq:ankle}
  \phi_2:=R^{\rho_n'}P^{-\pi_n'-\delta},
\end{equation}
with $\rho_n', \pi_n'$ given by \eqref{eq:rp2} and $\delta$ given by \eqref{eq:delta}.
Assuming \eqref{eq:in2} we note that if $\phi\leq \phi_2$ and $R\leq P^{n\ve/(3n-2)}$ then we would 
have a point on the major arcs if $\ve$ is sufficiently small in terms
of $\D$, which we have seen to be impossible. Hence the inequality $\phi\leq
\phi_2$ is also enough to ensure that the first term is $O(1)$.

When $v=2$ the exponent of $\phi$ is zero in the above and we will
have an overall contribution of $O(1)$ unless 
\begin{equation}
  \label{eq:cas2}
R\leq   P^{\frac{(3n-2)(2n-23+k)}{(n^2-6n+4)}+\delta}.
\end{equation}
We will return to this case shortly.
Recall the definitions \eqref{eq:sock}--\eqref{eq:ankle} of $\phi_0,\phi_1,\phi_2$.
It follows from the inequality $\phi_2<
\phi_1$ that $R^{\rho_n'+\rho_n}< P^{\pi_n'-\pi_n+2\delta}$.
We now employ the assumption \eqref{eq:paris} on the size of $\rho_n'+\rho_n$.
Combining the above we conclude that there is an overall contribution of 
$o(P^{10})$ to $I_{u,v}(k;\mathfrak{a})$ from all of the relevant values of
$R,\phi$, apart from those which satisfy the inequalities
$$
R< P^{{\frac{\pi_n'-\pi_n}{\rho_n'+\rho_n}}+2\delta},\quad \phi_2<\phi <\phi_1.
$$
But then the relevant point is forced to lie in the set $\m_0$
that was defined in the statement of the lemma. This is impossible,
and so completes the proof of Lemma~\ref{lem:v<2}.
\end{proof}

Our next result deals with the corresponding case in which $u=v=2$.
In this setting \eqref{eq:rp1} becomes
\begin{equation}
  \label{eq:rp1'}
  \rho_n=\frac{n(n-5)}{n^2-5n+2}, \quad \pi_n=
\frac{-2(n^2-(18-k)n-2)}{n^2-5n+2}.
\end{equation}
Define
\begin{equation}
  \label{eq:r''}
  \rho_n'':=
\frac{(3n-2)(2n-23+k)}{n^2-6n+4}
\end{equation}
and
\begin{equation}\label{eq:psi}
\psi_n:=\rho_n''\left(1+\frac{n}{8}-\frac{(4+n)\rho_n}{8}\right)
+n+\frac{k}{2}-\frac{(4+n)\pi_n}{8},
\end{equation}
for any $k$ and $n$, and recall the definition \eqref{eq:delta} of
$\delta$. 
Then we have the following result.

\begin{lemma}\label{lem:v=2}
Let $u=v=2$. Assume that $n_2\geq 6$.
Define $\m_0$ to be the set 
of $\al \in \m$ for which there exist coprime integers $0 \leq a <q$ such that 
$$
q\leq P^{\rho_{n_2}''+\delta}, \quad
\Big|\al-\frac{a}{q}\Big|\leq 
q^{-\frac{n_2-8}{n_2-4}}P^{-\frac{80-5n_2-4k}{n_2-4}+\delta}.
$$
Then 
$$
I_{2,2}(k;\mathfrak{n}\setminus \m_0)=o(P^{10}),
$$
for any $\mathfrak{n}\subseteq \m$, 
provided that \eqref{eq:in1} holds and 
\begin{equation}\label{eq:in3}
\psi_{n_2}\leq 10-\frac{1}{10}.
\end{equation}
\end{lemma}

\begin{proof}
We continue to write $\mathfrak{a}=\mathfrak{m}\setminus\m_0$ and 
$n=n_2$ throughout the proof, in order to improve the appearance of
our expressions. Our starting point is the proof of Lemma \ref{lem:v<2}, which on
passing to dyadic intervals via
\eqref{eq:foot}, shows that 
$I_{2,2}(k;\mathfrak{a})=o(P^{10})$ 
unless $\phi<\phi_1$, in the notation of \eqref{eq:knee}, 
and  the inequality \eqref{eq:cas2} holds for $R$.
This much is valid subject to \eqref{eq:in1}.

We now consider the effect of applying Lemma \ref{lem:B} in
\eqref{eq:foot} when $R$ and $\phi$ are in the remaining ranges, with
$u=v=2$. This gives
\begin{align*}
I_{2,2}(k;\mathfrak{a})
&\ll P^{n+\frac{k}{2}+2\ve} \max_{R,\phi} 
(R^2\phi)^{\frac{1}{2}} \big(R\phi+ (R\phi P^3)^{-1}\big)^{\frac{n}{8}}\\
&\ll  P^{2\ve}\max_{R,\phi} \big(
R^{1+\frac{n}{8}}\phi^{\frac{4+n}{8}}P^{n+\frac{k}{2}} +
R^{1-\frac{n}{8}}\phi^{\frac{4-n}{8}} P^{\frac{5n}{8}+\frac{k}{2}}\big).
\end{align*}
Taking $\phi<\phi_1$ and recalling the assumed inequality
\eqref{eq:cas2} for $R$ we see that the first term here is 
\begin{align*}
&\ll   
R^{1+\frac{n}{8}}(R^{-\rho_n}P^{-\pi_n+\delta})^{\frac{4+n}{8}}P^{n+\frac{k}{2}+2\ve}\\
&\ll   
R^{1+\frac{n}{8}-\frac{(4+n)\rho_n}{8}}
P^{n+\frac{k}{2}-\frac{(4+n)\pi_n}{8}+2\ve+(\frac{4+n}{8})\delta}\\
&\ll   P^{\psi_n+2\ve+(1+\frac{n}{8}-\frac{(4+n)\rho_n}{8}+\frac{4+n}{8})\delta},
\end{align*}
where $\psi_n$ is given by \eqref{eq:psi}. According to  
\eqref{eq:in3} this contribution is satisfactory. 
Turning to the second term in the above estimate for
$I_{2,2}(k;\mathfrak{a})$, we will have $O(P^{10-\ve})$ as an upper
bound for this quantity provided
that $\phi>\phi_3$, with 
$$
\phi_3:=R^{-\frac{n-8}{n-4}}P^{-\frac{80-5n-4k}{n-4}+\delta},
$$
since $n\geq 6$ by assumption. 

Our investigation has therefore allowed us to handle all $\al$ apart from
those for which $R\leq P^{\rho_n''+\delta}$ and $\phi<\phi_3$, where
$\rho_n''$ is given by \eqref{eq:r''}.
Such points are forced to lie on the set of arcs defined in $\m_0$. This therefore
completes the proof of Lemma \ref{lem:v=2}.
\end{proof}

The ideal scenario is when we can apply
Lemma \ref{lem:v=2} with 
$$
k=2n_1-3=23-2n_2,
$$
and we will find this is possible for 
certain ranges of $n_1,n_2$ such that $n_1+n_2=13$. When this comes to
pass it follows from \eqref{eq:rp1'}, \eqref{eq:r''} that
$\pi_{n_2}=2$ and $\rho_{n_2}''=0$, and furthermore,
$\psi_{n_2}=21/2-n_2/4$. One easily checks that the conditions in
\eqref{eq:in1} and \eqref{eq:in3} are satisfied for $n_2\geq 6$. 
Finally we note that $\m\setminus \m_0=\m$ in the statement of
Lemma \ref{lem:v=2} since clearly any element of $\m_0$ is forced
to lie on the major arcs. We may conclude as follows. 

\begin{lemma}\label{lem:v=2'}
Assume that $n_2\geq 6$.  Then 
we have 
$$
I_{2,2}(2n_1-3;\mathfrak{m})=o(P^{10}).
$$
\end{lemma}

We are now ready to apply this collection of estimates in our case by
case analysis of the minor arc integral $E$ in \eqref{eq:minor}.

\subsection{The case $n_1=1$}\label{s:1.12}

We will assume that $w\in \WW_n^{(2)}$ throughout this section. 
One of the ingredients in our treatment of this case 
is the use of ``pruning''. We will find it
convenient to sort the minor arcs into subsets
$$
\emptyset=\mathfrak{n}_3\subseteq \mathfrak{n}_2\subseteq \mathfrak{n}_1
\subseteq \n_0:= \m.
$$ 
Recall the definition \eqref{eq:delta} of $\delta$. We define $\n_1$ to be the set of $\al\in \m$ for
which there exists $a,q\in\ZZ$ such that $0\leq a<q\leq P^{17/24+2\delta}$ and $\gcd(a,q)=1$,
with 
\begin{equation}
  \label{eq:box1}
q^{\frac{42}{17}}P^{-4-\delta}\leq \Big|\al-\frac{a}{q}\Big|\leq q^{-1}P^{-\frac{37}{24}+2\delta},
\end{equation}
We denote by $\n_2$ the corresponding set of $\al \in \n_1$ with the
property that whenever \eqref{eq:box1} holds with $\gcd(a,q)=1$ and $0\leq a< q\leq P^{17/24+\delta}$, then
$$
q\leq P^{\frac{27}{50}}.
$$
We will write $E_i$ for the overall contribution to $E$ from
integrating over the set $\n_i\setminus \n_{i+1}$, for $i=0,1,2$.
Our task is to show that $E_i=o(P^{10})$ for each $i$.

To handle the case $i=0$ we begin as in \eqref{eq:split} and
\eqref{eq:Ivu} with $(u,v)=(4,4/3)$.  It easily follows that
$$
\int_{0}^1 |S_1(\al)|^4 \d \al \ll P^{2+\ve},
$$
on interpreting the integral as a sum over the solutions of
the equation $x_1^3+x_2^3=x_3^3+x_4^3$, with $x_i\ll P$, and
applying standard estimates for the divisor function. Hence we have
$$
E_0 \ll I_{4,\frac{4}{3}}(2;\m\setminus \n_1),
$$
in the notation of \eqref{eq:Iuv}.
When $(u,v)=(4,4/3)$, $k=2$ and $n_2=12$ we have 
$$
 \rho_{12}=\frac{30}{37}, \quad \pi_{12}=\frac{62}{37},\quad
\rho_{12}'=\frac{42}{17},\quad 
\pi_{12}'=4,
$$
in \eqref{eq:rp1} and \eqref{eq:rp2}. In particular \eqref{eq:paris}, \eqref{eq:in1} and
\eqref{eq:in2} are satisfied. 
Now it is easily to see  that $\m\setminus\n_1\subset
\n\setminus\m_0$, where $\m_0$ is as in the statement of Lemma~\ref{lem:v<2}, since for $\al\in \m_0$ we have 
\begin{align*}
q^{\frac{42}{17}}P^{-4-\delta}\leq 
\Big|\al-\frac{a}{q}\Big|\leq 
q^{-\frac{30}{37}}P^{-\frac{62}{37}+\delta}=q^{-1}q^{\frac{7}{37}}P^{-\frac{62}{37}+\delta}
&\leq q^{-1}P^{-\frac{37}{24}+2\delta}.
\end{align*}
It therefore follows from Lemma \ref{lem:v<2} that $E_0=o(P^{10})$, as
required.

Turning to the case $i=1$, we begin as above with the observation that 
$$
E_1 \ll I_{4,\frac{4}{3}}(2;\n_1\setminus \n_2),
$$
This time we appeal to Lemma \ref{lem:B}. On observing that 
$$
\Big|\al-\frac{a}{q}\Big|\leq q^{-1}P^{-\frac{37}{24}+2\delta}\leq q^{-1}P^{-\frac{3}{2}},
$$
for any $\al \in \n_1$, we deduce from the second part of this result
that
\begin{align*}
E_1 &\ll P^{\frac{25}{2}+2\ve} \max_{R,\phi} (R^2\phi)^{\frac{3}{4}}R^{-\frac{3}{2}}\min
\{1,(\phi P^3 )^{-\frac{3}{2}}\}
\ll P^{8+2\ve} \max_{R,\phi} \phi^{-\frac{3}{4}},
\end{align*}
where the maximum is over all $R,\phi>0$ such that 
$$
P^{\frac{27}{50}}< R\leq P^{\frac{17}{24}+2\delta}, \quad
R^{\frac{42}{17}}P^{-4-\delta}< \phi < R^{-1}P^{-\frac{37}{24}+2\delta}. 
$$
Taking the lower bounds for $\phi$ and $R$ that emerge from these
inequalities therefore implies that
\begin{align*}
E_1 &\ll P^{11+\frac{3\delta}{4}+2\ve} \max_{R} R^{-\frac{63}{34}} =o(P^{10}),
\end{align*}
on recalling that $\delta=10^{-4}$ from \eqref{eq:delta} and $\ve>0$
is arbitrary.

The key idea in our treatment of $E_2$ is to take
advantage of the fact that we have rather good control of the
$1$-dimensional exponential sum $S_1(\al)$ on suitable sets of ``major
arcs''. Recall the definition \eqref{eq:Aqa} of $\A=\A(A,B,C)$.
We will take $(A,B,C)=(24/50,1,35/24+2\delta)$, whence we may deduce
from Lemma \ref{lem:E} that
$$
S_1(\al)=S_1^*(\al)+O(P^{\frac{35}{48}+\delta+\ve}),
$$
for any $\al \in \A_{a,q}$, where $S_1^*(\al)$ is given by \eqref{eq:S*}.
It follows that
\begin{equation}\label{eq:improper0}
E_2
\ll
\int_{\n_2} |S_1^*(\al)S_2(\al)|\d\al
+P^{\frac{35}{48}+\delta+\ve} \int_{\n_2} |S_2(\al)|\d\al
=I_1+I_2,
\end{equation}
say.  We will show that $I_1$ and $I_2$ are both $o(P^{10})$.

Let us begin by analysing the first term in this bound. 
Now it follows from the second part of Lemma \ref{lem:E} and H\"older's inequality that 
\begin{align*}
I_1&\ll I_{4,\frac{4}{3}}(1,\n_2),
\end{align*}
in the notation of \eqref{eq:Iuv}.  A straightforward application of
Lemma \ref{lem:v<2} reveals that $I_{4,4/3}(1,\n_2\setminus
\n_*)=o(P^{10})$, where $\n_*$ is the set of $\al \in \n_1$ for which
there exists $a,q\in\ZZ$ such that $0\leq a<q\leq P^{17/48+2\delta}$ and $\gcd(a,q)=1$,
with 
$$
q^{\frac{42}{17}}P^{-3-\delta}\leq \Big|\al-\frac{a}{q}\Big|\leq q^{-\frac{30}{37}}P^{-\frac{68}{37}+\delta}.
$$
To estimate $I_{4,4/3}(1,\n_*)$ we employ the second part of Lemma
\ref{lem:B} in much the same way that we did in our analysis of $E_1$. This implies that
\begin{align*}
I_{4,\frac{4}{3}}(1,\n_*) &\ll P^{\frac{49}{4}+2\ve} \max_{R,\phi} (R^2\phi)^{\frac{3}{4}}R^{-\frac{3}{2}}\min
\{1,(\phi P^3 )^{-\frac{3}{2}}\},
\end{align*}
where the maximum is over all $R,\phi>0$ such that 
$$
R\leq P^{\frac{17}{48}+2\delta}, \quad
R^{\frac{42}{17}}P^{-3-\delta}< \phi <
R^{-\frac{30}{37}}P^{-\frac{68}{24}+\delta}, 
$$
with the inequalities $R\leq P^\D$ and $\phi \leq P^{-3+\D}$ not both
holding simultaneously. 
Taking the lower bound for $\phi$ we obtain the contribution
$$
\ll P^{\frac{31}{4}+2\ve} \phi^{-\frac{3}{4}} 
\ll P^{10+\frac{3\delta}{4}+2\ve}R^{-\frac{63}{34}}.
$$
This is $o(P^{10})$ if $R\geq P^{\delta}$. If on the other
hand $R<P^{\delta}\leq P^{\D}$ we must automatically have $\phi \geq P^{-3+\D}$,
  whence we still obtain a satisfactory contribution. This completes the treatment of 
$I_{4,4/3}(1,\n_*)$, and so that of $I_1$.

We now turn to the contribution from $I_2$ in \eqref{eq:improper0}.
Breaking the ranges for $q$ and $|\al-a/q|$ into dyadic intervals as
usual, and applying the second part of Lemma \ref{lem:B}, we have
$$
I_2\ll P^{12+\frac{35}{48}+\delta+2\ve} \max_{R,\phi} R^{\frac{1}{2}}\phi \min
\{1,(\phi P^3 )^{-\frac{3}{2}}\},
$$
where the maximum is over all $R,\phi$ such that 
$$
0<R\leq P^{\frac{27}{50}}, \quad  0<\phi < R^{-1}P^{-\frac{37}{24}+2\delta},
$$
with the inequalities $R\leq P^\D$ and $\phi \leq P^{-3+\D}$ not both
holding simultaneously.  If $\phi >P^{-3}$ then this is 
$$
\ll 
R^{\frac{1}{2}} P^{9+\frac{35}{48}+\delta+2\ve} \ll
P^{10-\frac{1}{1200}+\delta+2\ve}=o(P^{10}), 
$$
whereas if on the other hand $\phi\leq P^{-3}$, then the same basic conclusion holds.

Once taken all together, this therefore completes the treatment of the
minor arcs when $(n_1,n_2)=(1,12)$.

\subsection{The case $n_1=2$}

We will continue to assume that $w\in \WW_n^{(2)}$ throughout this section. 
In what follows we may assume that
$C_1$ does not take the shape
$a(b_1x_1+b_2x_2)^3$, for integers $a,b_1,b_2$, since otherwise
the resolution of Theorem \ref{main} is trivial.

Recall the manipulations in \eqref{eq:split} and
\eqref{eq:Ivu}. Taking $(u,v)=(4,4/3)$ it
follows from Lemma \ref{lem:hua_2} that the latter 
inequality is bounded  by $O(P^{5+\ve})$.
Thus we are led to estimate $I_{4,4/3}(5;\m)$, as given by
\eqref{eq:Iuv}.  We clearly have 
$$
\rho_{11}=\frac{44}{57}, \quad \pi_{11}=\frac{103}{57},\quad
\rho_{11}'=\frac{56}{31}, \quad \pi_{11}'=3,
$$
in \eqref{eq:rp1} and \eqref{eq:rp2}.
One easily checks that the conditions \eqref{eq:paris}, \eqref{eq:in1} and
\eqref{eq:in2} in the statement of
Lemma  \ref{lem:v<2} are satisfied with our choice of $u,v,k$ and
$n_2$. Recall the definition \eqref{eq:delta} of $\delta$. 
Define $\m_0$ to be the set 
of $\al \in \m$ for which there exist coprime integers $0 \leq a <q$ such that 
$$
q\leq P^{\frac{31}{67}+2\delta},\quad 
q^{\frac{56}{31}}P^{-3-\delta}\leq \Big|\al-\frac{a}{q}\Big|\leq q^{-\frac{44}{57}}P^{-\frac{103}{57}+\delta}.
$$
Then we may conclude from Lemma \ref{lem:v<2} that 
$
I_{4,\frac{4}{3}}(5;\m\setminus \m_0)=o(P^{10}).
$

It remains to deal with the contribution from $\al\in \m_0$.
We will use Lemma~\ref{lem:B} to handle this remaining range. 
Now it follows that
$$
\Big|\al-\frac{a}{q}\Big|\leq
q^{-1}q^{\frac{13}{57}}P^{-\frac{103}{57}+\delta}<q^{-1}P^{-\frac{3}{2}}.
$$ 
Recalling the definition \eqref{eq:Iuv} of  $I_{4,4/3}(5;\m_0)$,
we therefore deduce from the second part of Lemma \ref{lem:B} with $n=11$ that
\begin{align*}
  I_{4,\frac{4}{3}}(5;\m_0) &\ll P^{\frac{49}{4}+2\ve} \left(
\sum_{q}q^{-\frac{5}{6}} \int \min\{1,(|\theta| P^3)^{-\frac{11}{6}}\}
    \d\theta\right)^{\frac{3}{4}},
\end{align*}
where the integral is over $q^{56/31}P^{-3-\delta}\leq |\theta|\leq q^{-44/57}P^{-103/57+\delta}$
and the sum is over $q\leq P^{31/67+2\delta}$,
with the inequalities $q\leq P^\D$ and $|\theta| \leq P^{-3+\D}$ not both
holding simultaneously.
The contribution from $q\leq P^{\D}$ is therefore 
$$
\ll 
P^{10+2\ve-\frac{5\D}{8}} \left(\sum_{q\leq P^\D}q^{-\frac{5}{6}} \right)^{\frac{3}{4}} 
\ll P^{10+2\ve-\frac{\D}{2}},
$$
which is satisfactory.  Taking $|\theta|\geq q^{56/31}P^{-3-\delta}$,
we see that the corresponding contribution from 
$q>P^\D$ is 
$$
\ll 
P^{10+\frac{5\delta}{8}+2\ve}  \left(\sum_{q> P^\D}q^{-\frac{145}{62}} \right)^{\frac{3}{4}} 
\ll P^{10+\frac{5\delta}{8}-\D+2\ve}.
$$
This too is satisfactory, and so completes our analysis of the case $(n_1,n_2)=(2,11)$.

\subsection{The case $n_1=3$}

According to Lemma \ref{lem:curve} we may proceed under the assumption
that either $C_1$ is non-singular or else our cubic form $C$ splits off a
ternary norm form. In the latter case Theorem \ref{main'''} readily
ensures that $X(\QQ)\neq \emptyset$, and so we may focus our efforts
on the case $C_1$ is non-singular. 
We will assume that $w=(w_1,w_2)\in \WW_3^{(1)}\times \WW_{10}^{(2)}$,  throughout this section.

Our argument relies upon the same notion of pruning
that was put to good effect in  \S \ref{s:1.12}.
Let us define $\m_1$ to be the set of $\al \in \m$ for which there exists
$a,q\in\ZZ$ such that $0\leq a<q\leq P^{16/25}$ and $\gcd(a,q)=1$,
with 
$$
\Big|\al-\frac{a}{q}\Big|\leq q^{-1}P^{-\frac{143}{75}}.
$$
It will be convenient to refer to the set $\m\setminus\m_1$ as the set of ``proper minor
arcs'', and  $\m_1$ will be the set of ``improper minor
arcs''.  Let us write $E_{\mathrm{prop}}$ and $E_{\mathrm{improp}}$
for the corresponding contributions to $E$.

We begin by estimating $E_{\mathrm{prop}}$. 
Taking $u=v=2$ in \eqref{eq:split} and
\eqref{eq:Ivu}, and applying Lemma \ref{lem:C}, we deduce that 
$E_{\mathrm{prop}}\ll I_{2,2}(4;\m\setminus \m_1)$.
When $u=v=2$, $k=4$ and $n_2=10$ we have 
$$
 \rho_{10}=\frac{25}{26}, \quad \pi_{10}=\frac{21}{13},\quad
\rho_{10}''=\frac{7}{11},\quad 
\psi_{10}=\frac{839}{88}=9.53\ldots,
$$
in \eqref{eq:rp1'}, \eqref{eq:r''} and \eqref{eq:psi}.
Now it is easily to see  that $\m\setminus\m_1\subset
\m\setminus\m_0$, where $\m_0$ is as in the statement of Lemma
\ref{lem:v=2}, since for $\al\in \m_0$ we have 
\begin{align*}
\Big|\al-\frac{a}{q}\Big|\leq 
q^{-\frac{1}{3}}P^{-\frac{7}{3}+\delta}=q^{-1}q^{\frac{2}{3}}P^{-\frac{7}{3}+\delta}
&\leq q^{-1}P^{-\frac{21}{11}+2\delta}
<q^{-1}P^{-\frac{143}{75}},
\end{align*}
where $\delta$ is given by \eqref{eq:delta}.
On observing that $\psi_{10}$ satisfies \eqref{eq:in3}, and that
the inequalities in \eqref{eq:in1} are trivially satisfied, it therefore follows from 
Lemma \ref{lem:v=2} that $E_{\mathrm{prop}}=o(P^{10})$.

We now turn to the argument needed to control the overall contribution
to $E$ from the improper minor arcs, which we denote by
$E_{\mathrm{improp}}$. We select $(A,B,C)=(16/25,1, 82/75)$
in the definition \eqref{eq:Aqa} of $\A=\A(A,B,C)$.
It now follows from taking $n=3$ and $\sigma=-1$ in Lemma \ref{lem:F} that
$$
S_1(\al)-S_1^*(\al)\ll P^{\frac{A}{2}+2+\ve}+P^{2C+\ve}\ll 
P^{\frac{58}{25}+\ve},
$$
for any $\al\in \A_{q,a}$, where $S_1^*(\al)$ is given by \eqref{eq:S*}.
Hence 
$$
E_{\mathrm{improp}}
\ll
\int_{\m_1} |S_1^*(\al)S_2(\al)|\d\al
+P^{\frac{58}{25}+\ve} \int_{\m_1} |S_2(\al)|\d\al=I_1+I_2,
$$
say.  Our goal is to show that $E_{\mathrm{improp}}=o(P^{10})$.

We begin by handling the contribution from $I_2$.  
Using dyadic summation it follows from Lemma \ref{lem:B} that
\begin{align*}
I_2
\ll P^{10+\frac{58}{25}+2\ve} \max_{R,\phi} R^{\frac{3}{4}}\phi
\min\{1,(\phi P^3)^{-\frac{5}{4}}\}
&\ll P^{7+\frac{58}{25}+2\ve} R^{\frac{3}{4}}\ll P^{10-\frac{1}{5}+2\ve},
\end{align*}
where the maximum is over $R,\phi$ dictated by the definition of $\m_1$.
This is plainly satisfactory and so completes our treatment of $I_2$.

We now turn to an upper bound for $I_1$. Combining H\"older's
inequality with the second part of Lemma \ref{lem:F} gives
$$
|I_1|\leq
\left(\int_{\m_1} |S_1^*(\al)|^4\d\al
\right)^{\frac{1}{4}}\left(\int_{\m_1}
  |S_2(\al)|^{\frac{4}{3}}\d\al\right)^{\frac{3}{4}}\ll 
I_{4,\frac{4}{3}}(9;\m_1),
$$
in the notation of \eqref{eq:Iuv}. 
Let us dissect $\m_1$ into $\m_1^{a}\cup\m_1^{b}$, where
$\m_1^{a}$ is the set of $\al\in \m_1$
for which there exist coprime integers $0 \leq a <q$ such that 
$q\leq P^{16/25}$ and  $|\al-a/q|\leq q^{1/2}P^{-3+\delta}$, and 
$\m_1^{b}=\m_1\setminus \m_1^{a}$. An application of Lemma 
\ref{lem:B} yields
\begin{align*}
I_{4,4/3}(9;\m_1^b)
&\ll 
P^{10+\frac{9}{4}+2\ve}\left(
\sum_{q\leq P^{\frac{16}{25}}}q^{-\frac{2}{3}}\int_{q^{\frac{1}{2}}P^{-3+\delta}}^{q^{-1}P^{-\frac{143}{75}}}
(\theta P^3)^{-\frac{5}{3}}\d\theta
\right)^{\frac{3}{4}}\\
&\ll 
P^{10+\frac{9}{4}+2\ve}\left(
\sum_{q}q^{-\frac{2}{3}}P^{-5}
(q^{\frac{1}{2}}P^{-3+\delta})^{-\frac{2}{3}}
\right)^{\frac{3}{4}}\\
&\ll P^{10-\frac{\delta}{2}+2\ve}\log P,
\end{align*}
which is satisfactory.
Turning to $I_{4,4/3}(9;\m_1^a)$, we note that when $(u,v)=(4,4/3)$
and $k=9$ one has
$$
 \rho_{10}=\frac{5}{7}, \quad \pi_{10}=\frac{37}{21},\quad
\rho_{10}'=\frac{8}{7},\quad 
\pi_{10}'=3,
$$ 
in \eqref{eq:rp1} and \eqref{eq:rp2}. Since \eqref{eq:paris}, \eqref{eq:in1} and
\eqref{eq:in2} are evidently satisfied in Lemma~\ref{lem:v<2}, a modest pause for thought reveals that
$I_{4,4/3}(9;\m_1^a)=o(P^{10})$,  as required.

\subsection{The case $n_1=4$}

We follow the strategy of the preceding section. 
According to Lemma \ref{lem:surface} we may assume that either $C_1$ is
non-singular or else the surface $C_1=0$ contains precisely $3$
conjugate double points. In the latter case \eqref{eq:4_norm} implies
that our cubic form $C$ can be written as
$$
\Norm_{K/\QQ}(x_1\omega_1+\cdots+x_3\omega_3)+ax_4^2\Tr_{K/\QQ}(x_1\omega_1+\cdots+x_3\omega_3)+bx_4^3+
C_2(x_5,\ldots,x_{13}),
$$
for appropriate coefficients $\omega_1,\omega_2,\omega_3 \in K$ and
$a,b \in \ZZ$, and where $K$ is a certain cubic number field.
Setting $x_4=0$ we arrive at  a cubic form in $12$ variables which is exactly of the type 
considered in Theorem \ref{main'''}. Hence $X(\QQ)\neq \emptyset$ in
this case, and so  we are free to proceed under the assumption that $C_1$ is non-singular. 
Throughout this section we will take $w=(w_1,w_2)$ as our weight function, with
$w_1\in \WW_4^{(1)}$ and $w_2\in \WW_{9}^{(2)}$.

Let $\m_1$ be the set of $\al \in \m$ for which there exists
$a,q\in\ZZ$ such that $0\leq a<q\leq P^{21/50}$ and $\gcd(a,q)=1$,
with 
$$
\Big|\al-\frac{a}{q}\Big|\leq q^{-1}P^{-\frac{113}{50}}.
$$
As previously we define $\m\setminus\m_1$ to be the proper minor
arcs and  $\m_1$ to be the improper minor arcs, with the same notation   
$E_{\mathrm{prop}}, E_{\mathrm{improp}}$ for the corresponding
contributions to $E$.

We begin by estimating $E_{\mathrm{prop}}$. 
Taking $u=v=2$ in \eqref{eq:split} and
\eqref{eq:Ivu}, and applying Lemma \ref{lem:C}, we find
that $E_{\mathrm{prop}}\ll I_{2,2}(11/2;\m\setminus \m_1)$.
When $u=v=2$, $k=11/2$ and $n_2=9$ we have 
$$
 \rho_{9}=\frac{18}{19}, \quad \pi_{9}=\frac{67}{38},\quad
\rho_{9}''=\frac{25}{62},\quad 
\psi_{9}=9+\frac{15}{124}=9.12\ldots,
$$
in \eqref{eq:rp1'}, \eqref{eq:r''} and \eqref{eq:psi}.
Furthermore, it is easily checked that $\m\setminus\m_1\subset
\m\setminus\m_0$, where $\m_0$ is as in the statement of Lemma
\ref{lem:v=2}.   On observing that \eqref{eq:in3} and 
\eqref{eq:in1} are satisfied, it therefore follows from 
Lemma \ref{lem:v=2} that $E_{\mathrm{prop}}=o(P^{10})$.

It remains to estimate $E_{\mathrm{improp}}$, for which we select
$$
(A,B,C)=\Big(\frac{21}{50},1, \frac{37}{50}\Big)
$$
in the definition \eqref{eq:Aqa} of $\A$.
Taking $n=4$ and $\sigma=-1$ in Lemma \ref{lem:F} therefore gives
$$
S_1(\al)-S_1^*(\al)\ll P^{\frac{5C}{2}+\ve}+P^{\frac{5A}{6}+\frac{5}{2}+\ve}\ll 
P^{\frac{57}{20}+\ve},
$$
for any $\al\in \A_{q,a}$.  It now follows that 
$$
E_{\mathrm{improp}}
\ll
\int_{\m_1} |S_1^*(\al)S_2(\al)|\d\al
+P^{\frac{57}{20}+\ve} \int_{\m_1} |S_2(\al)|\d\al
=I_1+I_2,
$$
say.
We begin by handling the contribution from $I_2$.  
Using dyadic summation it follows from Lemma \ref{lem:B} that
\begin{align*}
I_2
\ll P^{9+\frac{57}{20}+2\ve} \max_{R,\phi} R^{\frac{7}{8}}\phi
\min\{1,(\phi P^3)^{-\frac{9}{8}}\}
&\ll P^{6+\frac{57}{20}+2\ve} R^{\frac{7}{8}}
\ll P^{9.2175+2\ve},
\end{align*}
where the maximum is over $R,\phi$ such that 
$$
0<R\leq P^{\frac{21}{50}},\quad 0<\phi<R^{-1}P^{-\frac{113}{50}}.
$$
This is plainly satisfactory and so completes our treatment of $I_2$.

We now turn to an upper bound for $I_1$.  
Since $C_1$ is assumed to be good as well as non-singular, we may
apply the second part of Lemma \ref{lem:F} with $k=3$ and $n=4$ to conclude that
$
I_1\ll I_{3,3/2}(9,\m_1).
$
As in the case $n_1=3$ we write $\m_1=\m_1^{a}\cup\m_1^{b}$, where now
$\m_1^{a}$ is the set of $\al\in \m_1$
for which there exist coprime integers $0 \leq a <q$ such that 
$q\leq P^{21/50}$ and  $|\al-a/q|\leq q P^{-3+\delta}$, and 
$\m_1^{b}=\m_1\setminus \m_1^{a}$. It follows from Lemma 
\ref{lem:B} that 
\begin{align*}
I_{3,\frac{3}{2}}(9;\m_1^b)
&\ll 
P^{12+2\ve}\left(
\sum_{q\leq P^{\frac{21}{50}}}q^{1-\frac{27}{16}}\int_{q P^{-3+\delta}}^{q^{-1}P^{-\frac{113}{50}}}
(\theta P^3)^{-\frac{27}{16}}\d\theta
\right)^{\frac{2}{3}}\\
&\ll 
P^{12+2\ve}\left(
\sum_{q\leq P^{\frac{21}{50}}}q^{-\frac{11}{16}}P^{-\frac{81}{16}}
(q P^{-3+\delta})^{-\frac{11}{16}}
\right)^{\frac{2}{3}}\\
&\ll P^{10-\frac{11\delta}{24}+2\ve}.
\end{align*}
To handle $I_{3,3/2}(9;\m_1^a)$ we note that when $(u,v)=(3,3/2)$
and $k=9$ we have 
$$
 \rho_{9}=\frac{3}{4}, \quad \pi_{9}=\frac{29}{16},\quad
\rho_{9}'=\frac{43}{25},\quad 
\pi_{9}'=3,
$$ 
in \eqref{eq:rp1} and \eqref{eq:rp2}. 
Lemma~\ref{lem:v<2} easily gives 
$I_{3,3/2}(9;\m_1^a)=o(P^{10})$,  as required.
This completes the treatment of the improper minor arcs when $(n_1,n_2)=(4,9)$.

\subsection{The case $n_1=5$}

An application of Lemma \ref{lem:3fold} reveals that we are free to assume
that $C_1$ defines a projective cubic hypersurface whose singular locus is either empty or finite.
Throughout this section we will take $w=(w_1,w_2)\in \WW_5^{(1)}\times
\WW_{8}^{(2)}$ as our weight function.

We let the improper minor arcs $\m_1$ be the set of $\al \in \m$ for which there exists
$a,q\in\ZZ$ such that $0\leq a<q\leq P^{\frac{11}{20}+\delta}$ and $\gcd(a,q)=1$,
with 
$$
\Big|\al-\frac{a}{q}\Big|\leq P^{-\frac{5}{2}+\delta},
$$
with $\delta$ given by \eqref{eq:delta}, and we let 
$\m\setminus\m_1$  be the proper minor
arcs.  As above, let $E_{\mathrm{prop}}, E_{\mathrm{improp}}$ denote
the corresponding contributions to $E$.

Taking $u=v=2$ in \eqref{eq:split} and
\eqref{eq:Ivu}, and applying Lemma \ref{lem:C} with $n=5$ and
$\sigma\leq 0$, we find
that $E_{\mathrm{prop}}\ll I_{2,2}(15/2;\m\setminus \m_1)$.
When $u=v=2$, $k=15/2$ and $n_2=8$ we have 
$$
 \rho_{8}=\frac{12}{13}, \quad \pi_{8}=\frac{22}{13},\quad
\rho_{8}''=\frac{11}{20},\quad 
\psi_{8}=9.55,
$$
in \eqref{eq:rp1'}, \eqref{eq:r''} and \eqref{eq:psi}.
Furthermore, it is easily checked that $\m\setminus\m_1\subset
\m\setminus\m_0$, where $\m_0$ is as in the statement of Lemma
\ref{lem:v=2}.   On observing that \eqref{eq:in1} and \eqref{eq:in3}
are satisfied, it therefore follows from 
Lemma \ref{lem:v=2} that $E_{\mathrm{prop}}=o(P^{10})$.

It remains to estimate $E_{\mathrm{improp}}$, for which we select
$$
(A,B,C)=\Big(\frac{11}{20}+\delta,0, \frac{1}{2}+\delta\Big)
$$
in the definition \eqref{eq:Aqa} of $\A$.
Taking $n=5$ and $\delta\leq 0$ in Lemma \ref{lem:F} therefore gives
$$
S_1(\al)-S_1^*(\al)\ll 
P^{3+\frac{5A}{3}+\ve}
+P^{3(A+C)+\ve}\ll P^{3.92},
$$
for any $\al\in \A_{q,a}$.  It now follows that 
$$
E_{\mathrm{improp}}
\ll
\int_{\m_1} |S_1^*(\al)S_2(\al)|\d\al
+P^{3.92} \int_{\m_1} |S_2(\al)|\d\al
=I_1+I_2,
$$
say.

We begin by handling the contribution from $I_2$.  
Using dyadic summation it follows from Lemma \ref{lem:B} that
\begin{align*}
I_2
\ll P^{3.92+\ve} \max_{R,\phi} R \phi
P^{8}\min\{1,(\phi P^3)^{-1}\}
\ll P^{8.92+\ve} R\ll P^{9.47+\delta +\ve},
\end{align*}
where the maximum is over $R,\phi$ such that 
$$
0<R\leq P^{\frac{11}{20}+\delta},\quad 0<\phi<P^{-\frac{5}{2}+\delta}.
$$
This is plainly satisfactory and so completes our treatment of $I_2$.

We now turn to an upper bound for $I_1$. 
Since $C_1$ is assumed to be good, we may
apply the second part of Lemma \ref{lem:F} with $k=12/5$ and $n=5$ to conclude that
$I_1\ll I_{12/5,12/7}(9;\m_1)$, in the notation of \eqref{eq:Iuv}.
When $(u,v)=(12/5,12/7)$, $k=9$ and $n_2=8$ we have 
$$
 \rho_{8}=\frac{4}{5}, \quad \pi_{8}=\frac{66}{35},\quad
\rho_{8}'=\frac{38}{11},\quad 
\pi_{8}'=3,
$$
in \eqref{eq:rp1} and \eqref{eq:rp2}.
Furthermore one easily checks that $k/u=15/4$ satisfies the
inequalities in \eqref{eq:in1}. It therefore follows from
Lemma \ref{lem:v<2} that $I_{12/5,12/7}(9;\m_1\setminus \m_2)=o(P^{10})$, where
$\m_2$ is the set of $\al\in \m_1$ for which there exist coprime
integers $0\leq a<q$ such that
\begin{equation}
  \label{eq:cut}
q\leq P^{\frac{11}{42}+2\delta},\quad q^{\frac{38}{11}}P^{-3-\delta}\leq
\Big|\al-\frac{a}{q}\Big|\leq q^{-\frac{4}{5}}P^{-\frac{66}{35}+\delta}. 
\end{equation}
Note that $q^{-4/5}P^{-66/35+\delta}\leq q^{-1}P^{-3/2}$. 
To handle $I_{12/5,12/7}(9;\m_2)$ we appeal to Lemma
\ref{lem:B}, deducing that 
\begin{align*}
I_{\frac{12}{5},\frac{12}{7}}(9;\m_2)
&\ll P^{8+\frac{15}{4}+2\ve} \max_{R,\phi} (R^2 \phi)^{\frac{7}{12}}
R^{-1}\min\{1,(\phi P^3)^{-1}\}\\
&\ll 
\max_{R,\phi} P^{5+\frac{15}{4}+2\ve}
R^{\frac{1}{6}}\phi^{-\frac{5}{12}}\\
&\ll P^{10+\frac{\delta}{2}+2\ve}R^{-\frac{14}{11}},
\end{align*}
on taking $\phi\geq R^{38/11}P^{-3-\delta}$.
Here the maximum is over the relevant $R,\phi$ determined by \eqref{eq:cut},
with the inequalities $R\leq P^\D$ and $\phi \leq P^{-3+\D}$ not both
holding simultaneously.
Now either $R>P^\D$ and this estimate is satisfactory, or else $R\leq
P^{\D}$ and it follows that we may actually take the lower bound
$\phi>P^{-3+\D}$ in the second term, giving instead
$$
\ll P^{10-\frac{5\D}{12}+2\ve} R^{\frac{1}{6}}\ll P^{10-\frac{\D}{4}+2\ve}.
$$
This too is satisfactory, and so completes the proof that
$I_1=o(P^{10})$, thereby completing the treatment of the minor arcs in
the case $(n_1,n_2)=(5,8)$.

\subsection{The case $n_1=6$}

We now come to the final case that we need to analyse
in our proof of Theorem \ref{main}.
We take $w=(w_1,w_2)\in  \WW_6^{(2)}\times \WW_{7}^{(2)}$, and seek 
an estimate for
$$
  M_6(P;\m):=\int_{\m}|S_1(\al)|^{2}\d\al.
$$
Note that the integral is now taken over the set of minor arcs, rather
than the entire interval $[0,1]$ as in \eqref{eq:moment}.
As usual we assume that $C_1$ and $C_2$ are good. 
We will show that 
\begin{equation}
  \label{eq:log}
 M_6(P;\m)\ll P^{9+\ve}. 
\end{equation}
Once in place, we may take 
$u=v=2$ in \eqref{eq:split} and \eqref{eq:Iuv} to conclude that
$
E\ll I_{2,2}(9;\m),
$
whence the desired conclusion is
given by Lemma \ref{lem:v=2'}.

It remains to establish \eqref{eq:log}. 
Let us consider the consequences of applying Lemma \ref{lem:D} to
estimate $M=M_6(P;\m)$, following the general approach in the proof of
Lemmas \ref{lem:v<2} and \ref{lem:v=2}.
We have 
$$
M
\ll P^3+P^\ve \max_{R,\phi}
\frac{\psi_H R^2 P^{11}}{H^{10}}F,
$$
where $\psi_H$ and $F$ are as in the statement of Lemma \ref{lem:D}
and $H\in [1,P]\cap \ZZ$ is arbitrary. Furthermore the maximum is over
$R,\phi$ such that
\eqref{eq:range} holds with any choice of $Q\geq 1$ that we care
to choose. We will take $Q=P^{3/2}$.
In our deduction of \eqref{eq:log} it will be convenient to allow the value of $\ve>0$
to take different values at different parts of the argument. 

We define $\phi_0:=R^{-2/11}P^{-2}$ and take
$$
H:=
\begin{cases}
\lfloor P^{\ve}\max\{1,(\phi R^{2}P^{2})^{\frac{1}{10}}\} \rfloor, &\mbox{if  $\phi>\phi_0$,}\\
\lfloor P^{\ve} R^{\frac{2}{11}} \rfloor, &\mbox{if  $\phi\leq \phi_0$.}
\end{cases}
$$
If we can show that $F\ll P^{\ve}$ with this
choice of $H$ then \eqref{eq:log} will follow.  It is clear that $H$ is an integer in the
interval $[1,P]$.

Suppose first that $\phi>\phi_0$. Then $\psi_H\ll \phi$ and it follows that
\begin{align*}
RH^3\psi_H
\ll P^\ve R\phi (1+\phi R^2P^2)^{\frac{3}{10}}
&\ll 1+ Q^{-1}P^{\frac{3}{5}+\ve}\ll 1,
\end{align*}
by \eqref{eq:range} and the fact that $Q=P^{3/2}$.
The third term in $F$ is 
\begin{align*}
\frac{H^{6}}{R^{3}(P^2\psi_H)^{2}}\ll 
\frac{P^{\ve}}{R^{3}P^{4}\phi^{2}}
(1+\phi R^2P^2)^{\frac{3}{5}}
&\ll
\frac{P^{\ve}}{R^{3}P^{4}\phi_0^{2}}
+\frac{P^{\ve}}{R^{\frac{9}{5}}P^{\frac{14}{5}}\phi_0^{\frac{7}{5}}}
\ll P^{\ve},
\end{align*}
which is satisfactory. 

Suppose now that $\phi\leq \phi_0$.  Then $\psi_H\ll (P^2H)^{-1}$ and it follows that
\begin{align*}
RH^3\psi_H
&\ll \frac{RH^2}{P^2}\ll \frac{RP^\ve}{P^2}(1+R^{\frac{2}{11}})\ll
1+\frac{Q^{\frac{13}{11}}P^\ve}{P^2}\ll 1.
\end{align*}
Turning to the third term in $F$, we find that
\begin{align*}
\frac{H^{6}}{R^{3}(P^2\psi_H)^{2}}\ll 
\frac{H^{8}}{R^{3}}
\ll
\frac{P^{\ve}}{R^{3}}
(1+R^{\frac{16}{11}})
\ll P^\ve,
\end{align*}
which is also satisfactory. This therefore completes the proof of
\eqref{eq:log}, and so our treatment of the case $(n_1,n_2)=(6,7)$.

\newpage
\appendix

 \setcounter{section}{-1}
\section{Appendice par J.-L. Colliot-Th\'el\`ene}

\begin{center}
\textsc{Groupe de Brauer non ramifi\'e des hypersurfaces
cubiques\\  singuli\`eres  (d'apr\`es P. Salberger)}
\end{center}

\bigskip

En r\'eponse \`a une question de R. Heath-Brown, P. Salberger en 2006
a indiqu\'e les grandes lignes de la d\'emonstration de l'\'enonc\'e suivant, qui
\'etend un r\'esultat
connu  dans le cas lisse (\cite{CT}).   Nous donnons le d\'etail de la d\'emonstration.
On utilise les notations usuelles
dans ce domaine. Pour $X$ un sch\'ema on note $\pic X= H^1_{\et}(X,\G_{m})$
son groupe de Picard et $\br X=H^2_{\et}(X,\G_{m})$ son groupe de Brauer.
Pour les propri\'et\'es usuelles de ces groupes, nous renvoyons le
lecteur \`a   \cite{CTSan}.

\begin{theorem-f}
Soit $k$ un corps de caract\'eristique z\'ero, $\kbar$ une cl\^oture alg\'ebrique,
$\g$ le groupe de Galois de $\kbar$ sur $k$.  Soit $X \subset {\bf P}^n_{k}$
une  intersection compl\`ete  g\'eo\-m\'etriquement int\`egre de dimension au moins $3$.
Supposons le lieu singulier vide
ou de codimension au moins \'egale \`a $4$ dans $X$.
Alors pour tout   $k$-mod\`ele projectif et   lisse $Y$ de $X$,
\begin{itemize}
\item[(a)]
le groupe de Picard de $\Y=Y\times_{k}\kbar$ est un module galoisien $\Z$-libre de type
fini
qui est stablement de pemutation;
\item[(b)]
on a $H^1(\g,\pic \Y)=0$;
\item[(c)]
on a $\br k    {\buildrel \simeq \over \rightarrow}   {\rm Ker} [ \br Y \to \br \Y]     .$
\end{itemize}
\end{theorem-f}

\begin{proof}[D\'emonstration]
Les anneaux locaux de $X$ en codimension au
moins 3 sont r\'e\-gu\-liers, donc factoriels (th\'eor\`eme
d'Auslan\-der-Buchsbaum, voir \cite[\S XI Thm. 3.13]{SGA 2}).
Comme $X$ est une intersection compl\`ete,
un  th\'eor\`eme de Grothendieck  \cite[\S XI Cor. 3.14]{SGA 2}, ex-conjecture de Samuel)
implique que tous les anneaux locaux de $X$ sont factoriels.
Ainsi les diviseurs de Weil sur $X$ sont tous des diviseurs de Cartier.
Ceci implique que pour tout ouvert $U \subset X$ la fl\`eche de restriction $\pic X \to
\pic U$ est surjective.
Cette fl\`eche est aussi injective. Soit en effet $D$ un diviseur sur $X$
qui est le diviseur d'une fonction rationnelle $f$ sur $U$.
Comme le compl\'ementaire de $U$ dans $X$ est de codimension
au moins 2 et que sur $X$ diviseurs de Weil et diviseurs de Cartier
co\"{\i}ncident, on conclut que $D$ est le diviseur de $f$ sur $X$.
Le m\^eme argument montre que toute fonction rationnelle sur $X$
d\'efinie et  inversible sur $U$ est d\'efinie et inversible sur $X$,
et comme $X$ est projectif et g\'eom\'etriquement int\`egre,
toute telle fonction est une constante, elle appartient \`a  $k^*$.

L'hypoth\`ese sur la codimension du lieu singulier est g\'eom\'etrique,
elle vaut pour $X_{K}$ pour toute extension $K/k$ de corps, par
exemple $\kbar/k$. Les m\^emes conclusions s'appliquent
donc \`a $\X$. En particulier
la fl\`eche de restriction
$\pic \X \to \pic \U $ est   un isomorphisme.

Par ailleurs, le Corollaire 3.7 de \cite[\S XII]{SGA 2} montre que la fl\`eche de
restriction
$$
\Z= \pic  {\bf P}^n_{k} \to \pic X
$$
qui envoie la classe de $1 \in \Z$
sur la classe du faisceau inversible ${\mathcal O}_{X}(1)$
est un isomorphisme. Il en est de m\^eme de
  $\Z= \pic  {\bf P}^n_{\kbar} \to \pic  \X $,
et l'action du groupe
de Galois sur $\Z    \simeq  \pic  \X $ est triviale.

En conclusion, sous les hypoth\`eses du th\'eor\`eme, le module galoisien
$\pic \U$ est le module galoisien trivial $\Z$   et
l'on a $\kbar^*   {\buildrel \simeq \over \rightarrow}   \kbar[U]^*$, o\`u $\kbar[U]$ est l'anneau des
fonctions d\'efinies sur $\U$. L'argument ci-dessus
montre aussi que l'application naturelle
$\pic U \to \pic \U$ est un isomorphisme.

D'une suite exacte bien connue (cf. \cite[p.386]{CTSan})
on d\'eduit que la fl\`eche naturelle
$\br k \to {\rm Ker} [\br U \to \br \U]$ est un isomorphisme.
Par des arguments standards sur le groupe de Brauer (puret\'e et
injection par passage d'une vari\'et\'e lisse \`a un ouvert)
un tel \'enonc\'e
implique le m\^eme \'enonc\'e pour toute $k$-vari\'et\'e projective et
lisse $k$-birationnelle \`a $X$ : c'est l'\'enonc\'e (c).

Soit $U \subset Y$ une $k$-compactification lisse de $U$
(le th\'eor\`eme de Hironaka assure l'existence d'une telle compactification).
On a alors la suite exacte de modules galoisiens
$$0 \to \Div_{\infty}\Y \to \pic \Y \to \pic \U  \to 0,$$
o\`u le groupe $\Div_{\infty}\Y $ est
le module de permutation sur les points de codimension 1 de $\Y $ en dehors de $\U $,
et le z\'ero
\`a gauche tient au fait que l'on a $\kbar^* \simeq \kbar[U]^*$.
La suite de modules galoisiens ci-dessus est scind\'ee, car tout groupe
$H^1(\g, P)$ \`a valeurs dans un module de permutation est nul (lemme de Shapiro).
Ainsi $\pic \Y$ est la somme directe de deux modules de permutation,
et est donc un module de permutation. Il en r\'esulte que pour tout autre mod\`ele
projectif
et lisse $Y'$, le module galoisien $\pic  \Y'$ est stablement de permutation
(\cite[Prop. 2.A.1 on p. 461]{CTSan}).
L'\'enonc\'e (b) en r\'esulte.
\end{proof}

\begin{cor-f}
Soit $X \subset {\bf P}^n_{k}$ une hypersurface cubique g\'eom\'etriquement in\-t\`egre
de dimension au moins $3$ qui n'est pas un c\^one.
Supposons le lieu singulier vide
ou de codimension au moins \'egale \`a $4$ dans $X$. Alors pour tout mod\`ele projectif
et lisse
$Y$ de $X$, on a $\br k  {\buildrel \simeq \over \rightarrow}  \br Y$.
\end{cor-f}

\begin{proof}[D\'emonstration] Au vu du th\'eor\`eme ci-dessus, il suffit de montrer
$\br \Y =0$.

Si l'hypersurface $X$ est lisse, on a $\br \X =0$  comme il est
\'etabli dans \cite{CT} sans restriction sur le degr\'e de
l'hypersurface.
Par l'invariance birationnelle du groupe de
Brauer pour les vari\'et\'es projectives et lisses ceci implique $\br \Y =0$.

Si  l'hypersurface cubique $X$ est singuli\`ere,
comme elle n'est pas un c\^one,
en utili\-sant les droites passant par un $\kbar$-point singulier on obtient
une \'equivalence birationnelle  de $\X$
avec l'espace projectif ${\bf P}^{n-1}_{\kbar}$, dont le groupe de Brauer est nul.
Par l'invariance birationnelle du groupe de Brauer
pour les vari\'et\'es projectives et lisses ceci implique $\br \Y =0$.
\end{proof}

\begin{remark}
Il serait int\'eressant de voir si le corollaire vaut pour
les hypersurfaces   de degr\'e sup\'erieur \`a 3.
C'est le cas lorsque les hypersurfaces sont lisses (\cite{CT}).
\end{remark}

\begin{remark}
La condition que la codimension du lieu singulier est au moins \'egale \`a 4
est n\'ecessaire. Dans \cite{CTSal} on trouve des hypersurfaces
cubiques g\'eom\'etrique\-ment int\`egres non coniques dans ${\bf P}^4_{k}$,
dont le lieu singulier est  un ensemble fini non vide de points,
et  qui admettent un mod\`ele projectif et lisse $Y$
avec $\br Y/\br k \neq 0$.
\end{remark}

\begin{remark}
Lorsque $k$ est un corps de nombres, le corollaire permet de conjecturer
que, sous les hypoth\`eses donn\'ees, le principe de Hasse et l'approxi\-mation faible
valent pour le lieu lisse de l'hypersurface cubique $X$.
\end{remark}

\renewcommand{\refname}{R\'ef\'erences}

\begin{flushright}
\scriptsize{
J.-L. Colliot-Th\'el\`ene}\\
\scriptsize{C.N.R.S., UMR 8628 CNRS-Universit\'e}\\
\scriptsize{Math\' ematiques, B\^atiment 425}\\
\scriptsize{Universit\'e   Paris-Sud}\\
\scriptsize{F-91405 Orsay}\\
\scriptsize{France}\\
\texttt{jlct@math.u-psud.fr}
\end{flushright}


\begin{thebibliography}{99}

\bibitem{baker} 
R.C. Baker, 
Diagonal cubic equations, II, {\em Acta Arith.} {\bf 53} (1989),  217--250.


\bibitem{BDL}
B.J. Birch, H. Davenport and D.J. Lewis,
The addition of norm forms. {\em Mathematika} {\bf 9} (1962), 75--82.


\bibitem{cubhyp-circle}
T.D. Browning,
Counting rational points on cubic hypersurfaces.
{\em Mathematika} {\bf 54} (2007), 93--112.


\bibitem{roth} 
T.D. Browning and D.R. Heath-Brown,
Integral points on cubic hypersurfaces.
{\em Analytic Number Theory: Essays in honour of Klaus Roth}, 
Cambridge University Press, to appear.


\bibitem{41}
T.D. Browning and D.R. Heath-Brown, Rational points on quartic
hypersurfaces. {\em J. reine angew. Math.}, to appear.


\bibitem{b-w}
{J.W. Bruce   and C.T.C. Wall},
{On the classification of cubic surfaces}.
{\em J. London Math. Soc.} {\bf 19} (1979), {245--256}.



\bibitem{cayley}
A. Cayley, A memoir on cubic surfaces. {\em Phil. Trans. Roy. Soc.}
{\bf 159} (1869), 231--326.

\bibitem{c-t-s}
J.-L. Colliot-Th\'el\`ene and P. Salberger,
Arithmetic on some singular cubic hypersurfaces.
{\em Proc. London Math. Soc.}  {\bf 58} (1989),  519--549.


\bibitem{ct2}
J.-L. Colliot-Th\'el\`ene, J.-J. Sansuc and P. Swinnerton-Dyer, 
Intersections of two quadrics and Ch\^atelet surfaces. II.  
{\em J. reine angew. Math.}  {\bf 374}  (1987), 72--168. 




\bibitem{coray76b}
D. F. Coray,
Algebraic points on cubic hypersurfaces.
{\em Acta Arith.}  {\bf 30} (1976), 267--296.

\bibitem{coray87}
D. F. Coray,
Cubic hypersurfaces and a result of Hermite.
{\em Duke Math. J.}  {\bf 54} (1987), 657--670.


\bibitem{c-t}
D. F. Coray and M. A. Tsfasman,
Arithmetic on singular Del Pezzo surfaces.
{\em Proc. London Math. Soc.}  {\bf 57} (1988),  25--87.


\bibitem{6}
D. F. Coray,
D.J. Lewis, N. Shepherd-Barron and P. Swinnerton-Dyer,
Cubic threefolds with six double points.  {\em 
Number theory in progress, Vol. 1 (Zakopane-Ko\'scielisko, 1997)},
63--74, de Gruyter, Berlin, 1999. 


\bibitem{dav-32} 
H. Davenport, 
Cubic forms in thirty-two variables.
{\em Philos. Trans. Roy. Soc. London. Ser. A} {\bf 251} (1959), 193--232. 


\bibitem{dav-16} 
H. Davenport, 
Cubic forms in sixteen variables.
{\em Philos. Trans. Roy. Soc. London. Ser. A} {\bf 272} (1963), 285--303. 


\bibitem{dav-book}
H. Davenport, {\em Analytic methods for Diophantine 
equations and Diophantine inequalities.} 2nd ed., CUP, 2005.

\bibitem{fowler} J. Fowler, 
A note on cubic equations. {\em Proc. Camb. Philos. Soc.} {\bf 58} 
(1962), 165--169.

\bibitem{f-m-t}
J. Franke, Y.I. Manin and Y. Tschinkel,
Rational points of bounded height on {F}ano varieties.
{\em Invent. Math.} {\bf 95} (1989), 421--435.

\bibitem{harris} 
J. Harris, {\em Algebraic Geometry}, Springer-Verlag, 1992.



\bibitem{hb-10} 
D.R. Heath-Brown,
Cubic forms in ten variables.
{\em Proc. London. Math. Soc.}
{\bf 47} (1983),  225--257.

\bibitem{14}D.R. Heath-Brown, Rational points on cubic
hypersurfaces. {\em Invent. Math.} {\bf 170} (2007), 199--230.

\bibitem{hooley1}
C. Hooley, On nonary cubic forms.  {\em J. reine angew. Math.}
{\bf 386} (1988), 32--98.


\bibitem{hooley2}
C. Hooley, On nonary cubic forms. II. {\em J. reine angew. Math.}
{\bf 415} (1991), 95--165.


\bibitem{kollar}
J. Koll\'ar,
Unirationality of cubic hypersurfaces.  
{\em J. Inst. Math. Jussieu}  {\bf 1}  (2002),  467--476. 


\bibitem{lewis}
D.J. Lewis,
Diophantine problems: solved and unsolved.
{\em Number theory and applications}, 103--121, Kluwer, Dordrecht, 1989.

\bibitem{manin}
Y.I. Manin, {\em Cubic forms}. 2nd ed.,
North-Holland Mathematical Library {\bf 4}, North-Holland
Publishing Co.,  1986. 

\bibitem{cayley'}
L. Schl\"afli, On the distribution of surfaces of the third order into
species. {\em Phil. Trans. Roy. Soc.}
{\bf 153} (1864), 193--247.

\bibitem{segre-b}
B. Segre,  On arithmetical properties of singular cubic surfaces.
{\em  J. London Math. Soc.} {\bf 19} (1944), 84--91.

\bibitem{segre}
C. Segre, Suella varieta cubica con dieci punti doppi dello spazio a
quattro dimensioni.
{\em Atta della R. Accademia della Scienae di Torino} {\bf 22}
(1886--87), 791--801.

\bibitem{skolem}
T. Skolem, 
Einige Bemerkungen \"uber die Auffindung der rationalen Punkte auf
gewissen algebraischen Gebilden. {\em Math. Zeit.} {\bf 63}  (1955),
295--312.  

\bibitem{vaughan}
R.C. Vaughan, 
{\em The Hardy--Littlewood method}, 2nd ed., CUP, 1997.


\bibitem{wooley}
T. Wooley,
On Weyl's inequality, Hua's lemma and exponential sums over binary forms.
{\em Duke Math. J.} {\bf 100} (1999), 373--423.


\end{thebibliography}

\begin{thebibliography}{99}


\bibitem{CT}
J.-L. Colliot-Th\'el\`ene,
The Brauer-Manin obstruction for complete intersections
of dimension $\geq 3$, appendix to \cite{PV}.

\bibitem{CTSal}
J.-L. Colliot-Th\'el\`ene and P. Salberger,
Arithmetic on some singular cubic hypersurfaces.
{\em Proc. London Math. Soc.}  {\bf 58} (1989),  519--549.

\bibitem{CTSan}
J.-L. Colliot-Th\'el\`ene et J.-J. Sansuc,
La descente sur les vari\'et\'es rationnelles, II,
{\em Duke Math. J.} {\bf 54} (1987) 375--492.

\bibitem{SGA 2}
A. Grothendieck,
Cohomologie locale des faisceaux coh\'erents et th\'eor\`emes de
Lefschetz locaux et globaux (SGA 2). {\em S\'eminaire de G\'eom\'etrie
Alg\'ebrique du Bois Marie, 1962. Augment\'e d'un expos\'e de
Mich\`ele Raynaud}, Documents Math\'ematiques (Paris) {\bf
  4}, Soci\'et\'e Math\'ematique de France, 2005.


\bibitem{PV}
B. Poonen and J.F. Voloch, Random Diophantine equations.
{\em Arithmetic of higher-dimensional algebraic varieties (Palo Alto,
CA, 2002)}, 175--184, Progr. Math. {\bf 226}, Birkh\"auser, 2004.


\end{thebibliography}
\end{document}